\documentclass[a4paper,12pt]{scrartcl}

\usepackage[utf8]{inputenc}
\usepackage[english]{babel}
\usepackage{amssymb}
\usepackage{amsmath}
\usepackage{latexsym}
\usepackage{amsthm}
\usepackage{eucal}
\usepackage{graphicx}
\usepackage{bbm}
\usepackage{verbatim}
\usepackage{epigraph}
\usepackage{mathtools}
\usepackage{authblk}
\usepackage[pdflatex]{crop}
\usepackage[bookmarks]{hyperref}
\usepackage{tikz}
\usepackage{tikz-cd}
\usepackage{wasysym}
\usepackage{adjustbox}
\usepackage{microtype}
\usepackage{chngcntr}
\allowdisplaybreaks
\usetikzlibrary{shapes.geometric,fit}

\setkomafont{disposition}{\normalfont\bfseries}
\newcommand{\auth}[0]{Tobias Fritz and Paolo Perrone}
\newcommand{\tit}[0]{Stochastic order on metric spaces and the ordered Kantorovich monad}
\newcommand{\kw}[0]{Stochastic order, Wasserstein distance, categorical probability, Giry monad, Kantorovich monad.}

\hypersetup{
pdfauthor={\auth},%
pdftitle={\tit},%
colorlinks, linktocpage=true, pdfstartpage=1, pdfstartview=FitV,%
breaklinks=true, pdfpagemode=UseNone, pageanchor=true, pdfpagemode=UseOutlines,%
plainpages=false, bookmarksnumbered, bookmarksopen=true, bookmarksopenlevel=1,%
hypertexnames=true, pdfhighlight=/O,%
urlcolor=black, linkcolor=black, citecolor=black, 
}
\numberwithin{equation}{section}

\theoremstyle{plain}
\newtheorem{thm}{Theorem}[subsection]
\newtheorem{lemma}[thm]{Lemma}
\newtheorem{prop}[thm]{Proposition}
\newtheorem{cor}[thm]{Corollary}

\newtheorem{deph}[thm]{Definition}
\newtheorem{prob}[thm]{Problem}
\theoremstyle{definition}
\newtheorem{remark}[thm]{Remark}
\newtheorem{eg}[thm]{Example}

\newcommand{\N}{\mathbb{N}}

\newcommand{\Q}{\mathbb{Q}}

\newcommand{\R}{\mathbb{R}}

\newcommand{\cat}[1]{{\mathsf{#1}}} 
\newcommand{\ar}[2][]{\arrow{#2}{#1}}
\newcommand{\uni}[2][]{\arrow[dashrightarrow]{#2}{#1}} 

\newcommand{\id}{\mathrm{id}} 

\newcommand{\supp}{\mathrm{supp}}

\DeclareMathOperator{\e}{\varepsilon}

\newcommand{\op}{\mathrm{op}}

\mathchardef\hy="2D
\DeclareMathOperator{\lomet}{\cat{L{\hy}OMet}}
\DeclareMathOperator{\lcomet}{\cat{L{\hy}COMet}}

\usepackage{todonotes}

\let\originalleft\left
\let\originalright\right
\renewcommand{\left}{\mathopen{}\mathclose\bgroup\originalleft}
\renewcommand{\right}{\aftergroup\egroup\originalright}

\tikzset{%
    bullet/.style={
       fill=black,
       circle,
       minimum width=1pt,
       inner sep=1pt
     },
     relation/.style={
       -,
       thick,
       shorten <=2pt,
       shorten >=2pt
     },
     function/.style={
       ->,
       thick,
       shorten <=2pt,
       shorten >=2pt
     },
     every fit/.style={
       ellipse,
       draw,
       inner sep=0pt
     }
}

\title{\vspace{-1cm}\tit}
\author[1]{Tobias Fritz\thanks{tfritz [at] pitp.ca}}
\affil[1]{\small Perimeter Institute for Theoretical Physics, Waterloo, ON (Canada)}
\author[2]{Paolo Perrone\thanks{pperrone [at] mit.edu}}
\affil[2]{Massachusetts Institute of Technology, Cambridge, MA (U.S.A.)}
\date{}

\begin{document}

\maketitle

\vspace{-1cm}

\begin{abstract}
\addcontentsline{toc}{section}{Abstract}
In earlier work, we had introduced the Kantorovich probability monad on complete metric spaces, extending a construction due to van Breugel. Here we extend the Kantorovich monad further to a certain class of \emph{ordered} metric spaces, by endowing the spaces of probability measures with the usual stochastic order. It can be considered a metric analogue of the probabilistic powerdomain. Our proof of antisymmetry of the stochastic order on these spaces is more general than previously known results in this direction.

The spaces we consider, which we call \emph{L-ordered}, are spaces where the order satisfies a mild compatibility condition with the \emph{metric itself}, rather than merely with the underlying topology. As we show, this is related to the theory of Lawvere metric spaces, in which the partial order structure is induced by the zero distances.
 
We show that the algebras of the ordered Kantorovich monad are the closed convex subsets of Banach spaces equipped with a closed positive cone, with algebra morphisms given by the short and monotone affine maps. Considering the category of L-ordered metric spaces as a locally posetal 2-category, the lax and oplax algebra morphisms are exactly the \emph{concave} and \emph{convex} short maps, respectively.
 
In the unordered case, we had identified the Wasserstein space as the colimit of the spaces of empirical distributions of finite sequences. We prove that this extends to the ordered setting as well by showing that the stochastic order arises by completing the order between the finite sequences, generalizing a recent result of Lawson. The proof holds on any metric space equipped with a closed partial order.

\end{abstract}

\tableofcontents

\newpage
\section{Introduction}\label{secintro}

The study of orders on spaces of probability measures induced by orders on the underlying space is of interest in many mathematical disciplines, and it is known under different names. In decision theory and in mathematical finance, one speaks of \emph{first-order stochastic dominance} of random variables~\cite{fishburn}. In probability theory, the common name is the \emph{usual stochastic order}~\cite{lehmann,stochastic-orders}. Most of the existing theory is specific to \emph{real-valued} random variables, where the order is an answer to the question of \emph{when a random variable is statistically larger than another one}. 
On a general ordered metric or topological space $X$, there are at least three ways to define such an order: given two probability measures $p,q$ on $X$,
\begin{enumerate}
 \item $p\le q$ if and only if $p$ assigns at most as much measure to any (say, closed) upper set as $q$ does;
 \item $p\le q$ if and only if there exists a coupling\footnote{Following the terminology of optimal transport~\cite{villani}, a \emph{coupling} of $p$ and $q$ is a probability measure on $X\times X$ such that its marginals are $p$ and $q$, respectively.} $r$ entirely supported on the set of ordered pairs, $\{(x,y)\in X\times X \mid x\le y\}$;
 \item $p\le q$ if and only if for all monotone functions $f:X\to\R$ of a certain class (for example, continuous), 
 \begin{equation*}
  \int f \, dp \le \int f \, dq .
 \end{equation*}
\end{enumerate}
A possible interpretation of the first condition is that \emph{the mass of the measure $p$ is overall placed lower in the order compared to $q$}. A possible interpretation of the second condition, in terms of optimal transport, is that there exists a transport plan from $p$ to $q$ such that mass moves at most upwards in the order. 
These two definitions are known to be equivalent for ordered Polish spaces by means of Strassen's theorem~\cite[Theorem~11]{strassen}, and for all ordered Hausdorff spaces by a result of Kellerer~\cite[Proposition~3.12]{kellerer}; see Section~\ref{ssecstord} for an example of how this can be applied. 
An interpretation of the third condition is that for any choice of utility function compatible with the order, the expected utility with measure $p$ will be at most the expected utility with measure $q$. The equivalence of this third definition with the other two has long been known in the literature for probability measures on $\R$. To the best of our knowledge, it was first stated for general completely regular topological spaces by Edwards~\cite{edwards} with monotone and bounded lower semi-continuous functions.

While it is easy to see that the stochastic order over any partially ordered space is reflexive and transitive, antisymmetry seems to be a long-standing question~\cite{lawson,hll}. For the case of \emph{L-ordered} metric spaces, which is a concept introduced in this paper, we will show that antisymmetry indeed holds in Section~\ref{ssecstord}. This widely generalizes the antisymmetry result of~\cite{hll}.\footnote{The first author has proved in follow-up work an even more general result, \cite{antisymmetry}.}

From the point of view of category theory, it was first shown by Giry~\cite{giry}---building on ideas of Lawvere~\cite{early}---that much of the structure of the space of probability measures on a given underlying space can be captured in terms of a \emph{monad}~\cite[Chapter~VI]{maclane}, which Giry called \emph{probability monad}. The first probability monad on a category of ordered spaces, namely continuous domains, was defined seven years later by Jones and Plotkin~\cite{jones-plotkin}, and called \emph{probabilistic powerdomain}.
In more recent years, Keimel~\cite{keimel} studied another probability monad for ordered spaces, the \emph{Radon monad} on compact ordered spaces. He gave a complete characterization of its algebras, proving that they are precisely the compact convex subsets of locally convex topological vector spaces, with the order specified by a closed positive cone. 

A categorical treatment of probability measures on (unordered) \emph{metric} spaces was initiated by van~Breugel~\cite{breugel} with the introduction of the \emph{Kantorovich monad}, a probability monad utilizing the Kantorovich-Wasserstein distance, and coming in one version for compact metric spaces and one for complete $1$-bounded metric spaces. This construction was extended by us to all complete metric spaces~\cite{ours_kantorovich}, and shown to arise in a natural way from finite constructions which involve no measure theory. In this paper, we extend our Kantorovich monad to partially ordered complete metric spaces. We show how to make the interpretation of the order in terms of ``moving the mass upward'' precise in terms of a colimit characterization of the order, generalizing a result of Lawson~\cite{lawson}. We also prove that the algebras for the ordered Kantorovich monad are exactly the closed convex subsets of Banach spaces, equipped with a closed positive cone. Moreover, we give a categorical characterization of \emph{convex maps} between ordered convex spaces as exactly the \emph{oplax} morphism of algebras. 

Ordered metric spaces are closely related to \emph{Lawvere metric spaces}~\cite{lawvere,seriously}, which are generalizations of metric spaces to asymmetric distances. Such objects already incorporate a partial order structure in terms of zero distances.
A treatment of probability monads on Lawvere metric spaces, and the related Kantorovich duality theory, has been initiated by Goubault-Larrecq~\cite{gl}. 
In this paper we work for the most part with ordinary metric spaces; however, the duality theory and the interplay between metric and order can be interpreted in terms of Lawvere distances, as we show in Appendix~\ref{secindlmet}.

\paragraph{Summary.} 
In Section~\ref{secomet} we define the relevant categories of ordered metric spaces. In~\ref{ssecstord} we give the definition of the usual stochastic order, and that of ordered Wasserstein spaces.

In Section~\ref{colimit_char}, we show that the ordered Wasserstein space satisfies a colimit characterization, thanks to a density result (Proposition~\ref{orderdensity}), in analogy with the colimit characterization of unordered Wasserstein spaces given in~\cite[Theorem~3.3.7]{ours_kantorovich}. In~\ref{ssecproperness}, we prove some useful technical results unrelated to the order structure, such as that the marginal map $\nabla : P(X\otimes Y) \to P(X) \otimes P(Y)$ is proper.

In Section~\ref{secLord} we define and study a particular class of ordered spaces, which we call \emph{L-ordered spaces}, in which the order is compatible with the metric in a particular way. In~\ref{sseckantLord} we show that this property allows us to express the stochastic order in terms of Kantorovich duality (Theorem~\ref{Ldual}), and in~\ref{ssecantis} we prove, using this Kantorovich duality, that the order is antisymmetric (Corollary~\ref{antisymmetry}).

In Section~\ref{secordp}, we introduce and study the monad structure on the functor $P$ assigning to every ordered complete metric space its ordered Wasserstein space, resulting in the \emph{ordered Kantorovich monad}. 
In~\ref{ssecordms} we prove (Theorem~\ref{bimonord}) that the formation of product distributions and marginals equips the ordered Kantorovich monad with a \emph{bimonoidal structure}, just like in the unordered case~\cite[Section~5]{ours_bimonoidal}.

In Section~\ref{secordalg} we prove that the algebras of the ordered Kantorovich monad are precisely the closed convex subsets of ordered Banach spaces (Theorem~\ref{poscone}). 
The structure maps, as in the unordered case, are given by integration, and in~\ref{strictlymonotone} we show that these maps are strictly monotone, fully generalizing a result that has long been known in the real-valued case (Proposition~\ref{epev}). In~\ref{ssechigher} we show that, if one considers the category of ordered metric spaces as a locally posetal 2-category, then the algebra adjunction of the monad $P$ can be strengthened to an isomorphism of partial orders, and the Hahn-Banach separation theorem can be phrased as stating that $\R$ is a 2-categorical coseparator in the 2-category of $P$-algebras (Definition~\ref{defcosep} and Corollary~\ref{cosepalg}).
In~\ref{ssecoplax} we show, again using the 2-categorical perspective, that the lax and oplax morphisms of algebras are precisely the concave and convex maps (Theorem~\ref{laxoplax}).

Appendix~\ref{appmlift} develops some general results on a property of short maps between metric spaces which we call the \emph{metric lifting property}. Besides the applications in \ref{ssecproperness}, we will have other uses for this machinery in upcoming work.

In Appendix~\ref{secindlmet} we show that the concept of L-ordered space introduced in Section~\ref{secLord} can be interpreted in terms of Lawvere metric spaces, and that the L-distance is part of an adjunction between ordered metric spaces and Lawvere metric spaces (Theorem~\ref{LadjS}).

\paragraph{Remark.}
An earlier version of this paper is part of the second author's PhD thesis, available at \href{http://paoloperrone.org/phdthesis.pdf}{http://paoloperrone.org/phdthesis.pdf}.

\section{Probabilities on ordered metric spaces}\label{secomet}

\subsection{Preliminaries}

\begin{deph}
 An \emph{ordered metric space} is a metric space $X$ equipped with a partial order relation whose graph $\{\le\}\subseteq X\times X$ is closed. 
\end{deph}

So if we have sequences $\{x_i\}$ and $\{y_i\}$ in $X$ converging to $x$ and $y$, respectively, and such that $x_i\le y_i$ for all $i$, then necessarily $x\le y$. In analogy with the monoidal category $\cat{Met}$ from~\cite[Section~2.1]{ours_kantorovich}, we put:

\begin{deph}
 The symmetric monoidal category $\cat{OMet}$ has:
 \begin{itemize}
  \item As objects, ordered metric spaces;
  \item As morphisms, monotone, short maps (also called non-expanding, or 1-Lipschitz), i.e.~functions $f:X\to Y$ such that for all $x,x'\in X$, 
  $$
  d\big( f(x),f(x') \big) \le d(x,x');
  $$
  \item As monoidal structure $\otimes$, the cartesian product $X\times Y$ equipped with the $\ell^1$-sum of the metrics
  $$
  d\big( (x,y) , (x',y') \big) = d(x,x') + d(y,y') ,
  $$
  the product order, and together with the obvious symmetric monoidal structure isomorphisms.
 \end{itemize}
\end{deph}

There exists a faithful and essentially surjective forgetful functor $U:\cat{OMet}\to\cat{Met}$ with a left adjoint (the discrete order).

We are mainly interested in \emph{complete} metric spaces. 

\begin{deph}
 The category $\cat{COMet}$ is the full subcategory of $\cat{OMet}$ whose objects are ordered metric spaces which are complete as metric spaces.   
\end{deph}

Our construction of the ordered Kantorovich monad will take place on $\cat{COMet}$, while some of our analytical results will hold on all of $\cat{OMet}$.

As in~\cite{ours_kantorovich}, we are interested in probability measures of finite first moment. If $X$ is a metric space, a Borel measure $p$ on $X$ has finite first moment if for every short map $f:X\to\R$, the integral
$$
\int f\,dp
$$
exists and is finite. See~\cite[Section~2.3]{ours_kantorovich} for more details on this notion.

\begin{deph}
 Let $X\in\cat{Met}$. We denote by $PX$ the set of Radon probability measures on $X$ of finite first moment.
\end{deph}

A central theme of this work is the celebrated \emph{Kantorovich duality}~\cite[Chapter~5]{villani}. The following formulation can be obtained from~\cite[Theorem~5.10]{villani} together with~\cite[Particular Case~5.4]{villani})

\begin{thm}[Kantorovich duality]\label{kantorovichduality}
 Let $X$ be a Polish space. Let $p$ and $q$ be Radon probability measures on $X$, and let  $c:X\otimes X\to\R_+$ be a lower-semicontinuous function satisfying the triangle inequality. Then we have an equality:
  \begin{equation}
  \inf_{r\in\Gamma(p,q)} \int_{X\times X} c(x,y) \,dr(x,y) = \sup_{f} \, \left( \int_X f dq - \int_X f \, dp  \right) ,
 \end{equation}
 where the infimum is taken over the space $\Gamma(p,q)$ of \emph{couplings} between $p$ and $q$, and where $f:X\to\R$ varies over functions which have finite integral with both measures $p$ and $q$, and such that $f(y)-f(x)\le c(x,y)$ for all $x,y\in X$.
\end{thm}

The more specific form of Kantorovich duality that we will use in this work is the following:

\begin{cor}\label{kantorovichduality2}
 Let $X$ be a complete metric space. Let $c:X\otimes X\to\R_+$ be a lower-semicontinuous function bounded above by the distance, and which satisfies the triangle inequality. Let $p,q\in PX$. Then there is an equality
 \begin{equation}
  \inf_{r\in\Gamma(p,q)} \int_{X\times X} c(x,y) \,dr(x,y) = \sup_{f} \, \left( \int_X f dq - \int_X f \, dp  \right) ,
 \end{equation}
 where $f:X\to\R$ varies over functions such that $f(y)-f(x)\le c(x,y)$ for all $x,y\in X$.
\end{cor}

\begin{proof}
	First of all, the support of the Radon probability measures $p$ and $q$ is separable (see for example~\cite[Theorem~II.2.1]{partha}). Denote now by $\tilde{X}$ the union of the supports of $p$ and $q$. As the union of two closed separable sets, it is closed and separable as well. Therefore it is also complete, and hence Polish.
 Moreover, the supremum is taken over maps $f$ such that $f(y)-f(x)\le c(x,y)\le d(x,y)$, i.e.~they are short. Short maps can always be extended from a closed subset to the whole space, e.g.~in the following way: given $f:\tilde{X}\to\R$, we define $f':X\to\R$ to be 
 \begin{equation*}
  f'(x) := \sup_{y\in\tilde{X}} \big( f(y) - d(x,y) \big).
 \end{equation*} 
 Therefore the supremum over such short maps $f:X\to\R$ can be equivalently taken over maps $f:\tilde{X}\to\R$. 
 We can then apply Theorem~\ref{kantorovichduality} to get:
 \begin{align*}
  &\inf_{r\in\Gamma(p,q)} \int_{X\times X} c(x,y) \,dr(x,y) = \inf_{r\in\Gamma(p,q)} \int_{\tilde{X}\times \tilde{X}} c(x,y) \,dr(x,y)\\
  &= \sup_{f:\tilde{X}\to\R} \left( \int_{\tilde{X}} f dq - \int_{\tilde{X}} f \, dp  \right) = \sup_{f:X\to\R} \left( \int_X f dq - \int_X f \, dp  \right) .
 \end{align*} 
 Since $p$ and $q$ have finite first moment, the integral of such short $f$ with both measures always exists. 
\end{proof}

We now equip the space $PX$ with the Kantorovich-Wasserstein distance, or \emph{earth mover's distance}. This is given either by
\begin{equation*}
 d(p,q) := \inf_{r\in\Gamma(p,q)} \int_{X\times X} d(x,y) \,dr(x,y) ,
\end{equation*}
or equivalently, using Corollary~\ref{kantorovichduality2}, by
\begin{equation*}
 d(p,q) := \sup_{f:X\to\R} \left( \int_X f \, dq - \int_X f\, dp  \right),
\end{equation*}
where the supremum is taken over all the short maps $X\to\R$. 

It is well-known that if $X$ is complete (resp.~separable, compact), then $PX$ with the metric above is complete (resp.~separable, compact) as well \cite{villani,hitch}. More details of how the space $PX$ is constructed from a categorical point of view can be found in~\cite{ours_kantorovich}.

\subsection{The stochastic order}\label{ssecstord}

\begin{deph}\label{defstochord}
 Let $X\in \cat{OMet}$. For any $p,q\in PX$, the \emph{stochastic order} relation $p\le q$ holds if and only if there exists a coupling of $p$ and $q$ entirely supported on the graph $\{\le\}\subseteq X\otimes X$. 
\end{deph}

More explicitly, a coupling $r$ of $p$ and $q$ is a probability measure $r$ on $X\times X$ whose marginals are $p$ and $q$ respectively. The measure $r$ is supported on $\{\le\}\subseteq X\otimes X$ if and only if for every $x,x'$ such that $x\nleq x'$ and every open neighborhoods $U$ of $x$ and $U'$ of $x'$, we have that $r(U\times U')=0$.
This is a standard notion, see for example \cite{hll}. A possible interpretation, as sketched in the introduction, is that the mass of $p$ can be moved so as to form the distribution $q$ in a way such that every unit of mass is only moved upwards in the order (or not at all).

 As sketched in the introduction, the stochastic order can be defined in several equivalent ways. The following equivalence result is a special case of~\cite[Proposition~3.12]{kellerer}, which holds even for arbitrary topological spaces equipped with a closed partial order\footnote{Such a space is automatically Hausdorff~\cite[Proposition~2]{nachbin}.}.

\begin{thm}[Kellerer]\label{kellererthm}
 Let $X\in \cat{OMet}$, and let $p,q\in PX$. Then $p\le q$ if and only if $p(C) \le q(C)$ for every closed upper set $C\subseteq X$.
\end{thm}

In contrast to Definition~\ref{defstochord}, transitivity of the order relation is immediate from this alternative characterization.

Upon applying Theorem~\ref{kellererthm} to the order itself and then again to the opposite order, it also follows that $p \leq q$ holds if and only if $p(U) \leq q(U)$ for all \emph{open} upper sets $U$.

\section{Colimit characterization}
\label{colimit_char}

It is well-known that finitely supported measures are dense in the Wasserstein space~\cite{hitch,villani}, and therefore also empirical distributions of finite sequences. A possible interpretation is that the Kantorovich distance is a limit distance given by optimal transport of smaller and smaller finite partitions of the amount of mass to move. 
In~\cite{ours_kantorovich}, we used this type of reasoning to construct both the functor $P$ and its monad structure purely combinatorially without resort to any measure theory. (The measure only enters in proving the equivalence with the standard notion of probability measure.) Here, we prove that the \emph{order} structure of $PX$ also arises in this way, as the closure of the order between the finite empirical sequences. In other words, we prove the following alternative characterization of the stochastic order, generalizing~\cite[Theorem~4.8]{lawson}: 
$p\le q$ if and only if $p$ and $q$ can be approximated arbitrarily well by empirical distributions of finite sequences $\{x_i\}$ and $\{y_i\}$, such that up to permutation, $x_i\le y_i$ for all $i$, meaning that to obtain $q$ from $p$, each unit of mass must be moved upward in the order. 

We construct the spaces of finite sequences in a functorial way in~\ref{ssecpowers}. Then we define the empirical distribution map as a natural transformation in~\ref{ssecempdist}, and prove the above order density result in~\eqref{ssecorddens}. The constructions are analogous to those in~\cite{ours_kantorovich}, to which we refer for a more detailed treatment of the metric notions.

\subsection{Power functors}\label{ssecpowers}

The power functors for the Kantorovich monad on $\cat{CMet}$ were introduced and studied in~\cite{ours_kantorovich}. Here, we generalize some of this treatment to $\cat{COMet}$.

\begin{deph}
 Let $X\in\cat{OMet}$ and $N$ be a finite set. We denote by $X^N$ the \emph{$N$-fold power} of $X$, defined as follows:
 \begin{itemize}
  \item Its elements are functions $N\to X$, or equivalently tuples $(x_n)_{n\in N}$ of elements of $X$ indexed by elements of $N$;
  \item Its metric is defined to be:
  \begin{equation}
   d \big( (x_n)_{n\in N}, (y_n)_{n\in N} \big) := \dfrac{1}{|N|} \sum_{n\in N} d(x_n,y_n) ;
  \end{equation}
  \item Its order is the product order: $(x_n) \le (y_n)$ if and only if $x_n \le y_n$ for all $n\in N$.
 \end{itemize}
\end{deph}

This assignment is functorial in $X$. For this construction to be useful, we want to take a limit for $N$ large, or rather a colimit. So we need to make the powers functorial in $N$ as well, using the appropriate indexing category. To this end, define the monoidal category $\cat{FinUnif}$ as
\begin{itemize}
 \item Objects are nonempty finite sets;
 \item Morphisms are functions whose fibers all have the same cardinality (so in particular they are all surjective);
 \item The monoidal product is the cartesian product of finite sets. 
\end{itemize}
More details about this category can be found in \cite{ours_kantorovich}. Since this category is equivalent to a small category, we need not worry about size issues.

Given $X\in\cat{OMet}$, the powers $X^-$ form a functor $\cat{FinUnif}^\op\to\cat{OMet}$. In fact, consider a map $\phi: M\to N$ with fibers of uniform cardinality. Define the map $X^\phi:X^N\to X^M$ given by composition with $\phi$,
\begin{equation}
X^\phi(x_n)_{n\in N} := (x_{\phi(m)})_{m\in M}.
\end{equation}

\begin{prop}
$X^\phi$ is an isometric order embedding. 
\end{prop}

\begin{proof}
We know from \cite[Lemma 3.1.3]{ours_kantorovich} that $X^\phi$ is an isometric embedding. 
For the order part, first of all, $(x_{\phi(m)})_{m\in M} \le (y_{\phi(m)})_{m\in M}$ if an only if for all $m\in M$, $x_{\phi(m)}\le y_{\phi(m)}$. Since $\phi$ is surjective, this is equivalent to $x_{n}\le y_{n}$ for all $n\in N$, which in turn means exactly that $(x_n)_{n\in N}\le (y_n)_{n\in N}$.
\end{proof}

Since all these constructions are natural, we then have a functor $(-)^{(-)}:\cat{FinUnif}^\op\otimes\cat{OMet}\to\cat{OMet}$, or by currying, we consider equivalently the functor $(-)^{(-)}:\cat{FinUnif}^\op\to[\cat{OMet},\cat{OMet}]$. 
The curried functor is strongly monoidal, where the monoidal structure of the functor category $[\cat{OMet},\cat{OMet}]$ is given by functor composition. 
If we restrict to \emph{complete} ordered metric spaces, the powers $X^N$ are complete as metric spaces as well, and we get a strong monoidal functor $(-)^{(-)}:\cat{FinUnif}^\op\to[\cat{COMet},\cat{COMet}]$.

\subsection{Empirical distribution}\label{ssecempdist}

\begin{deph}
 Let $X\in\cat{OMet}$. Then a tuple $(x_n)_{n\in N} \in X^N$ induces an element of $PX$: the \emph{empirical distribution}
\begin{equation}
 \dfrac{1}{|N|} \sum_{n\in N} \delta_{x_n} .
\end{equation}
\end{deph}

This assignment forms a short, monotone map $i_N:X^N\to PX$, natural in $X$ and in $N$.

Forming the empirical distribution is not an isometric or order embedding. However, it is one up to permutation, as the following results show.

\begin{lemma}[Splitting Lemma]\label{splitting}
 Let $X\in \cat{OMet}$. Let $(x_n)\in X^N$ and $(y_m)\in X^M$. Then $i_N(x_n) \le i_M(y_m)$ if and only if there exist a set $K$ and maps $\phi:K\to N$ and $\psi:K\to M$ in $\cat{FinUnif}$ such that $X^\phi(x_n) \le X^\psi (y_m)$.
\end{lemma}

\begin{proof}
 The homonymous statement in \cite[Proposition IV-9.18]{continuous} implies\footnote{The stochastic order considered there coincides with ours if one takes the topology on $X$ to be given by the open upper sets of $X$.} in particular that for two finitely supported measures (``simple valuations'') $\zeta=\sum_n r_n \delta_{x_n}$ and $\xi=\sum_m s_m \delta_{y_m}$, we have $\zeta\le \xi$ if and only if there exists a matrix of entries $t_{n,m}\in[0,\infty)$ such that:
 \begin{enumerate}
  \item\label{ptsc} $t_{n,m}> 0$ only if $x_n\le y_m$;
  \item\label{rowsc} $\sum_m t_{n,m} = r_n$;
  \item\label{colsc} $\sum_n t_{n,m} \le s_m$.
 \end{enumerate}
 In our case, $\zeta := i_N(x_n)$ and $\xi := i_M(y_m)$ are normalized, so condition \ref{colsc} can be strengthened to an equality. Since all $r_n$ and $s_m$ are rational, the $t_{n,m}$ can also be chosen to be rational if they exist\footnote{Recall that if a finite system of linear inequalities with rational coefficients has a real solution, then it also has a rational solution. One way to see this is to note that the set of solutions is a convex polyhedron defined by linear inequalities, and use the fact that Farkas' lemma holds both over $\R$ and over $\Q$.}. By finiteness, we can find a common denominator $d$ for all its entries, so that the matrix $(t_{n,m})$ can be written as the empirical distribution of an element of $X^{M\otimes N\otimes D}$, where $|D|=d$. Therefore we can fix $K=M\otimes N\otimes D$. Conditions \ref{rowsc} and \ref{colsc} together with naturality of the empirical distribution imply that we can find the desired maps $\phi$ and $\psi$, and condition \ref{ptsc} then says that $X^\phi(x_n) \le X^\psi (y_m)$. 
\end{proof}

\begin{cor}\label{splittingcor}
 Let $X\in \cat{OMet}$. Let $(x_n),(y_n)\in X^N$. Then $i_N(x_n) \le i_N(y_n)$ if and only if there exists a permutation $\sigma:N\to N$ such that for each $n\in N$, $x_n \le y_{\sigma(n)}$.
\end{cor}

\begin{proof}
The ``if'' direction is clear. For ``only if'', we assume $i_N(x_n) \le i_N(y_n)$. Then the matrix $(t_{n,m})$ constructed as in the proof of Lemma~\ref{splitting} is bistochastic, and therefore a convex combination of permutations by the Birkhoff--von Neumann theorem. Choosing any permutation which appears in such a convex combination works, thanks to property~\ref{ptsc}.
\end{proof}

Moreover, as stated in \cite[Proposition 3.2.3]{ours_kantorovich}, we have an analogous statement for the metric:
\begin{equation}\label{permutationmetric}
 d\big( i_N(x_n),i_N(y_n)\big) = \min_{\sigma\in S_N} d\big( (x_n),(y_{\sigma(n)}) \big) .
\end{equation}

\subsection{Order density}\label{ssecorddens}

In \cite{ours_kantorovich} we proved that for a complete metric space $X$, the complete metric space $PX$ is the colimit of the $X^N$ for $N \in \cat{FinUnif}^\op$, taken in the category $\cat{CMet}$ of complete metric spaces. We now prove the analogous statement in the ordered case.

We start with a density result, which still holds for general ordered metric spaces. 

\begin{prop}\label{orderdensity}
 Let $X\in \cat{OMet}$ and $p\le q$ in $PX$. Then there exists a sequence $\{N_j\}_{j\in\N}$ in $\cat{FinUnif}$, and $\{\bar{p}_j\},\{\bar{q}_j\}$ such that:
 \begin{itemize}
  \item $\bar{p}_j,\bar{q}_j\in X^{N_j}$ for all $j$;
  \item $i(\bar{p}_j)\to p$ and $i(\bar{q}_j)\to q$ in $PX$;
  \item $\bar{p}_j\le \bar{q}_j$ in the order of $X^{N_j}$ for all $j$.
 \end{itemize}
\end{prop}

In other words, the order of $PX$ is the closure of the order induced by the image of all the empirical distributions. Or equivalently, any two probability measures in stochastic order can be approximated arbitrarily closely by uniform finitely supported measures which are also stochastically ordered.

This result generalizes Lawson's recent~\cite[Theorem~4.8]{lawson}, who has also found applications of this type of result to generalizations of operator inequalities.

\begin{proof}
 Consider the set
 \begin{equation}
  I(X) := \bigcup_{N\in\cat{FinUnif}} X_{|N|} \subseteq PX ,
 \end{equation}
 where $X_{|N|}$ is the quotient of $X^N$ under permutations of the components\footnote{See the \emph{symmetrized power functors} of~\cite{ours_kantorovich}.}. This set is dense in $PX$~\cite[Theorem~3.3.3]{ours_kantorovich}, and we equip it with the smallest ordering relation which makes the canonical maps $X^N \to I(X)$ monotone; by Lemma~\ref{splitting}, this is equivalently the restriction of the stochastic order from $PX$ to $I(X)$.

Let now $p,q \in PX$, and suppose $p\le q$. 
 By Corollary \ref{kellererthm}, there exists a joint $r$ on $X\otimes X$ supported on $\{\le\}$ with marginals $p$ and $q$. Now consider $\{\le\} \subseteq X^2$ and the subset $I(\{\le\})$ of $P(\{\le\})$, which is dense by~\cite[Theorem~3.3.3]{ours_kantorovich}. This means that for every $\e>0$, we can find a $\bar r\in I(\{\le\})$ such that $d(r,\bar r) < \e$. Let now $\bar{p},\bar{q}$ be the marginals of $\bar r$. Since the marginal projections are short~\cite[Proposition~5.10]{ours_bimonoidal}, we have $d(p,\bar{p})<\e$ and  $d(q,\bar{q}) <\e$. Moreover, again by Corollary~\ref{kellererthm}, since $\bar r$ is supported on $\{\le\}$, we also have $\bar{p} \le \bar{q}$. 
 By taking $\e$ smaller and smaller, we get the desired sequence.
\end{proof}

\begin{cor}
 $PX$ is the colimit of $X^{(-)} : \cat{FinUnif} \to \cat{COMet}$, with colimit components given by the empirical distribution maps $i_N : X^N \to PX$.
\end{cor}

\begin{proof}
 By~\cite[Theorem~3.3.7]{ours_kantorovich}, we already know that upon forgetting the order structure, we obtain a colimit in $\cat{CMet}$. Therefore, we only need to show that given any commutative cocone indexed by $N$, i.e.\ made up of triangles
 \begin{equation}\label{cocone}
  \begin{tikzcd}
   X^N \ar[swap]{dr}{f_N} \ar{r}{\phi} & X^{M} \ar{d}{f_{M}} \\
   & Y
  \end{tikzcd}
 \end{equation}
 where each cocone component $f_N$ is monotone, then also the unique short map $u$ in
 \begin{equation}\label{unicocone}
  \begin{tikzcd}
   X^N \ar[swap]{d}{i} \ar{dr}{f_N} \\
   PX \arrow[dashrightarrow,swap]{r}{u} & Y
  \end{tikzcd}
 \end{equation}
 is monotone. Now let $p\le q$. By Proposition \ref{orderdensity}, we can find sequences $\{N_j\}$ in $\cat{FinUnif}$, and $\{\bar{p}_j\},\{\bar{q}_j\}$ such that:
 \begin{itemize}
  \item $\bar{p}_j,\bar{q}_j\in X^{N_j}$ for all $j$;
  \item $i(\bar{p}_j)\to p$ and $i(\bar{q}_j)\to q$;
  \item $\bar{p}_j\le \bar{q}_j$ in the order of $X^{N_j}$ for all $j$.
 \end{itemize}
 Since $u$ is short, it is in particular continuous. By the commutativity of \eqref{unicocone},
 \begin{align*}
  u(p) =  u \big(\lim_j i(\bar{p}_j) \big) = \lim_j u\circ i(\bar{p}_j) = \lim_j f_{N_j} (\bar{p}_j) ,
 \end{align*}
 and just as well $u(q) =  \lim_j f_{N_j} (\bar{q}_j)$.
 Now for all $j$, $\bar{p}_j\le \bar{q}_j$, and since all the $f_{N_j}$ are monotone, $f_{N_j}(\bar{p}_j)\le f_{N_j}(\bar{q}_j)$. 
 By the closure of the order on $Y$, we then have that 
 \begin{align*}
  u(p) = \lim_j f_{N_j} (\bar{p}_j) \le \lim_j f_{N_j} (\bar{q}_j) = u(q) ,
 \end{align*}
 which means that $u$ is monotone. 
\end{proof}

\subsection{Properness of the marginal map}\label{ssecproperness}

This subsection is a technical development unrelated to the order structure, but we include it here since Corollary~\ref{mincouplingexists} below will be useful later. 
A map is \emph{proper} if preimages of compact sets are compact.\footnote{Note that some authors call this a \emph{perfect map}.}
The material in this section states that for every two complete metric spaces $X$ and $Y$, the marginalization map $P(X\otimes Y) \to PX \otimes PY$ is proper; while this could also have been derived using Prokhorov's theorem as in~\cite[Section~I.4]{villani} together with the separability of the support, we present a different approach based directly on compact approximation without resort to separability. In upcoming work, we will present an analogous proof in a more sophisticated situation, namely to prove properness of the monad multiplication $PPX \to PX$.

Here is the main statement.

\begin{thm}\label{deltaproper}
 Let $X,Y\in\cat{CMet}$. The map $\Delta:P(X\otimes Y)\to PX\otimes PY$ is proper. 
\end{thm}

In order to prove the theorem, we need some technical results, part of which can be found in Appendix~\ref{appmlift}. In particular, we need the notion of \emph{metric lifting} (Definition~\ref{defmetriclifting}), which is a sort of metric analogue of the homotopy lifting property. 

\begin{lemma}\label{deltametriclifting}
Let $X,Y\in\cat{CMet}$. Then the marginal map $\Delta : P(X\otimes Y) \to PX \otimes PY$ has the metric lifting property of Definition~\ref{defmetriclifting}.
\end{lemma}

Before proving this, we first derive an analogous statement for finite powers. Denote by $\Delta^N:(X\otimes Y)^N\to X^N\otimes Y^N$ the map
\begin{equation}
 \{(x_n,y_n)\}_{n\in N} \longmapsto \big( \{x_n\}_{n\in N}, \{y_n\}_{n\in N} \big),
\end{equation}
which has the nice feature that it is an isometry.

\begin{lemma}\label{deltamliftpow}
 Let $\bar{r} = \{\bar{r}_{n}\}_{n\in N} \in (X\otimes Y)^N$ and $(\bar{p},\bar{q}) = (\{\bar{p}_{n}\}_{n\in N},\{\bar{q}_{n}\}_{n\in N}) \in X^N\otimes Y^N$. Suppose that
\begin{equation}\label{ddlesse}
 d\big( (i\otimes i)\circ \Delta^N(\bar{r}), (i(\bar{p}),i(\bar{q})) \big) < C 
\end{equation}
in $X^N \otimes Y^N$. Then there exists $\bar{s} \in (X\otimes Y)^N$ such that $\Delta^{N}(\bar{s})=(\bar{p}\circ\sigma,\bar{q}\circ\sigma')$ for some permutations $\sigma,\sigma'\in S_N$, and $d(\bar{r},\bar{s}) < C$. 
\end{lemma}

\begin{proof}[Proof of Lemma~\ref{deltamliftpow}]
 Denote explicitly $\bar{r}_n:=(x_n,y_n)$ for all $n\in N$.
 By the formula~\eqref{permutationmetric}, condition~\eqref{ddlesse} is equivalent to
 \begin{equation*}
  \min_{\sigma,\sigma'\in S_{N}} \dfrac{1}{|N|} \sum_{n\in N}  \big(  d(x_{n},p_{\sigma(n)}) + d(y_{n},q_{\sigma'(n)}) \big) < C ,
 \end{equation*}
 which means that there exist $\bar\sigma,\bar\sigma'\in S_{N}$ such that 
 \begin{equation}\label{minbarsigma}
  \dfrac{1}{|N|} \sum_{n\in N} \big( d(x_{n},p_{\bar\sigma(n)}) + d(y_{n},q_{\bar\sigma'(n)}) \big) < C.
 \end{equation}
 Let now
 \begin{equation*}
  \bar{s}:=\{\bar{s}_{n}\} := \{(p_{\bar\sigma(n)},q_{\bar\sigma'(n)})\} \in (X\otimes Y)^N.
 \end{equation*}
 Then~\eqref{minbarsigma} implies that
 \begin{align*}
  d(\bar{r},\bar{s}) &= \dfrac{1}{|N|} \sum_{n\in N} \big( d(x_{n},p_{\bar\sigma(n)}) + d(y_{n},q_{\bar\sigma'(n)}) \big) < C.
 \end{align*} 
\end{proof}

In order to reduce Lemma~\ref{deltametriclifting} to this case, it helps to use a general density argument, Lemma~\ref{metricliftinggeneral}, stated and proven in Appendix~\ref{appmlift}.

\begin{proof}[Proof of Lemma~\ref{deltametriclifting}]

By Proposition~\ref{orderdensity}, we have that $I(X\otimes Y)$ is dense in $P(X\otimes Y)$ and $I(X)\otimes I(Y)$ is dense in $PX\otimes PY$. Lemma~\ref{deltamliftpow} says equivalently that given $r\in I(X\otimes Y)$ and $(p,q)\in I(X)\otimes I(Y)$ with $d(\nabla r,(p,q))< C$, there exists $s\in I(X\otimes Y)$ such that $\nabla(s)=(p,q)$ and $d(r,s) < C$. These are exactly the hypotheses of Lemma~\ref{metricliftinggeneral}, with $\nabla: P(X\otimes Y) \to PX\otimes PY$ in place of $f:X\to Y$, and $I(X\otimes Y)$ and $I(X)\otimes I(Y)$ in place of $D$ and $E$, respectively. Then by Lemma~\ref{metricliftinggeneral}, $\nabla$ satisfies the metric lifting.
\end{proof}

We can now prove the theorem with the help of Lemma~\ref{pointtoproper}, proven in Appendix~\ref{appmlift}, which says that whenever the metric lifting property holds, it is enough to check compactness of preimages of single points.

\begin{proof}[Proof of Theorem~\ref{deltaproper}]
 By Lemma~\ref{pointtoproper}, we only need to show that preimages of points are compact.
 Let $(p,q)\in PX\otimes PY$. Then by density, for every $\e>0$ there exist $p_{\e}\in PX$ and $q_{\e}\in PY$ with compact (even finite) support $K_{\e}$ and $H_{\e}$, respectively, and such that $d(p,p_{\e})<\e/4$ and $d(q,q_{\e})<\e/4$. By Lemma~\ref{deltametriclifting}, for every $r\in \Delta^{-1}(p,q)$ we can find some $r_{\e}$ such that $d(r,r_{\e})<\e/2$ and $\Delta(r_{\e})=(p_{\e},q_{\e})$. Now $r_{\e}$ must be supported on (a subset of) $K_{\e}\times H_{\e}$, which is itself compact, and which does not depend on $r$ varying in $\Delta^{-1}(p,q)$. In other words, the whole $\Delta^{-1}(p,q)$ is contained within an $\e/2$-neighborhood of $P(K_{\e}\times H_{\e})$. By compactness, for every $\e>0$, $P(K_{\e}\times H_{\e})$ can be covered by a finite number of balls of radius $\e/2$. Then $\Delta^{-1}(p,q)$ can be covered by a finite number of balls of radius $\e$, i.e.~it is totally bounded. 
 Since $\Delta$ is continuous, $\Delta^{-1}(p,q)$ is closed. Therefore $\Delta^{-1}(p,q)$ is compact.
\end{proof}

In particular, we have shown that an optimal coupling exists, generalizing~\cite[Theorem~I.4.1]{villani} to the non-separable case:

\begin{cor}\label{mincouplingexists}
 Let $X\in\cat{CMet}$. Given $p,q\in PX$, the set of couplings $\Gamma(p,q)=\Delta^{-1}(p,q)$ is compact. Therefore the infimum appearing in the Kantorovich duality formula is actually a minimum:
 \begin{equation}
  \min_{r\in\Gamma(p,q)} \int_{X\times X} c(x,y) \,dr(x,y) = \sup_{f} \: \left( \int_X f dq - \int_X f \, dp  \right) ,
 \end{equation}
\end{cor}

This does not seem to be a consequence of classical results like~\cite[Proposition~(1.2)]{kellerer}, which refer to the weak topology, while our compactness result refers to the topology induced by the Wasserstein distance; the latter is generally strictly finer in the non-separable case. However, thanks to compactness and Hausdorffness in both topologies, we can conclude \emph{a posteriori} that the Wasserstein topology on the subset $\Gamma(p,q)$ is equal to the weak topology.

\section{L-ordered spaces}\label{secLord}

\subsection{Definition}

In this section we study a compatibility condition between the metric and the order which is stronger than closedness. While closedness is a merely topological property, we now introduce a property that depends nontrivially \emph{on the metric itself}.

\begin{deph}\label{defLorder}
 Let $X$ be an ordered metric space. We say that $X$ is \emph{L-ordered} if for every $x,y\in X$ the following conditions are equivalent:
 \begin{itemize}
  \item $x\le y$;
  \item for every \emph{short, monotone} function $f:X\to\R$, $f(x)\le f(y)$.
 \end{itemize}
\end{deph}

It is easy to see that this condition implies closedness of the order. Also, the condition is similar to the following property of the metric, which all spaces have:
\begin{equation}
 d(x,y) = \sup_{f:X\to\R} f(x) - f(y),
\end{equation}
where the supremum is taken over all short maps. 
The intuition is that on L-ordered spaces, short functions, which are the functions that are enough to determine the metric, are also enough to determine the order. 

For all ordered metric spaces, the first condition in Definition~\ref{defLorder} implies the second. 
The converse does not always hold, as the following counterexample shows. 

\begin{eg}
\label{not_Lordered_example}
 Consider the space $X$ consisting of four different disjoint sequences $\{a_n\},\{b_n\},\{c_n\},\{d_n\}$ and two extra points $a,d$ with:
\begin{itemize}
 \item $\{a_n\}$ tending to $a$, with $d(a_n,a)=\frac{1}{n}$ for all $n\ge 1$;
 \item $\{d_n\}$ tending to $d$, with $d(d_n,d)=\frac{1}{n}$ for all $n\ge 1$;
 \item $a_n\le b_n$ for all $n\ge 1$,
 \item $c_n\le d_n$ for all $n\ge 1$,
 \item $d(b_n,c_n)=\frac{1}{n}$ for all $n\ge 1$;
 \item All other distances equal to 1;
 \item No points other than those indicated above are related by the order, in particular $a\nleq d$.
\end{itemize}
It is important to note that the two sequences $\{b_n\}$ and $\{c_n\}$ are not Cauchy. In fact, the only two nontrivial Cauchy sequences are $\{a_n\}$ and $\{d_n\}$, so that the space is even complete and the order is closed, resulting in $X\in\cat{COMet}$; Figure~\ref{not_Lordered_figure} provides an illustration.

Now consider a short, monotone function $X\to\R$. We have that:
\begin{align*}
 f(a) &= \lim_{n\to\infty} f(a_n) \le \lim_{n\to\infty} f(b_n) \\
 &= \lim_{n\to\infty} f(c_n) \le \lim_{n\to\infty} f(d_n) = f(d),
\end{align*}
however, $a\nleq d$. Therefore $X$ is not L-ordered.
\end{eg}

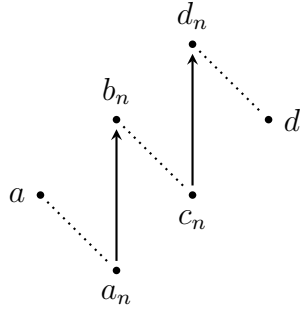
\begin{figure}
\begin{center}
  \begin{tikzpicture}[baseline=(current  bounding  box.center),>=stealth]
  \node[bullet,label=left:$a$] (a) at (2,0) {};
  \node[bullet,label=below:$a_n$] (an) at (3,-1) {};
  \node[bullet,label=above:$b_n$] (bn) at (3,1) {};
  \node[bullet,label=below:$c_n$] (cn) at (4,0) {};
  \node[bullet,label=above:$d_n$] (dn) at (4,2) {};
  \node[bullet,label=right:$d$] (d) at (5,1) {};
  \draw[relation,dotted] (a) -- (an);
  \draw[relation,dotted] (bn) -- (cn);
  \draw[relation,dotted] (dn) -- (d);
  \draw[function] (an) -- (bn);
  \draw[function] (cn) -- (dn);
 \end{tikzpicture}
\end{center}
\caption{The space $X \in \cat{COMet}$ constructed in Example~\ref{not_Lordered_example}. The dotted lines indicate distances $\frac{1}{n}$, and the arrows denote order relations.}
\label{not_Lordered_figure}
\end{figure}

Many ordered metric spaces of interest in mathematics are L-ordered. In particular, as we will show, L-ordered metric spaces are precisely the spaces which are embeddable into an ordered Banach space (see Remark~\ref{lordembed}). So far we have not yet been able to find any non-L-ordered space which is also compact:

\begin{prob}
	Is every compact ordered metric space automatically L-ordered?	
\end{prob}

In the following, we denote by $\lomet$ and $\lcomet$ the full subcategories of $\cat{OMet}$ and $\cat{COMet}$, respectively, consisting of the L-ordered (complete) metric spaces.

\subsection{Kantorovich duality for the order structure}\label{sseckantLord}

L-ordered spaces allow to study the stochastic order using Kantorovich duality. In particular, on an L-ordered space, we have a dual characterization of the order in terms of duality to Lipschitz functions. We want to prove the following theorem.

\begin{thm}\label{Ldual}
 Let $X\in\lcomet$ and $p,q\in PX$. Then $p\le q$ if and only if for every short monotone map $f:X\to\R$,
 \begin{equation}\label{Ldualorder}
  \int f \, dp \le \int f \, dq .
 \end{equation}
\end{thm}

The direction ``only if'' of the theorem does not need $L$-orderedness (see the proof below). Indeed, we can interpret $L$-orderedness as the fact that short monotone maps are \emph{enough} to detect the stochastic order. 
We will prove the theorem using ordinary Kantorovich duality with respect to a suitable cost function. The latter will be the following quantity, which is sensitive to both the metric and the order structure.

\begin{deph}\label{defLdist}
Let $X$ be an ordered metric space. The \emph{L-distance} is
\begin{equation}
 d_L(x,y) := \sup_{f:X\to\R} \left( f(x) - f(y) \right),
\end{equation}
where the supremum is taken over all short, monotone maps.
\end{deph}

This quantity can be interpreted as a \emph{Lawvere metric} compatible with the order, see Appendix~\ref{secindlmet}. More intuitively, the L-distance is to short monotone maps as the usual distance is to short maps, as the following remark shows. 

\begin{remark}\label{smL}
 Let $X$ and $Y$ be ordered metric spaces, and let $f:X\to Y$ be short and monotone. Then 
 \begin{align*}
  d_L\big( f(x), f(x') \big) &= \sup_{g:Y\to\R} \left( g(f(x)) - g(f(x')) \right) \\
  &\le \sup_{h:X\to\R} \left( h(x) - h(x') \right) = d_L(x,x'),
 \end{align*}
 where as usual $g$ and $h$ range over short, monotone maps.
\end{remark}

Here are some useful properties satisfied by $d_L$, which make it suitable for Kantorovich duality.

\begin{prop}\label{dlproperties}
 Let $X$ be an ordered metric space, not necessarily L-ordered. The L-distance satisfies the following properties:
 \begin{enumerate}
  \item For all $x,y\in X$ such that $x\le y$, we have $d_L(x,y)=0$. In particular, $d_L(x,x)=0$. 
  \item If (and only if) $X$ is L-ordered, $d_L(x,y)=0$ implies $x\le y$ for all $x,y$ in $X$.
  \item $d_L$ satisfies the triangle inequality: for every $x,y,z\in X$,
  $$
  d_L (x,z) \le d_L(x,y) + d_L(y,z) ,
  $$
  so that $d_L$ is a Lawvere metric (see Section~\ref{secindlmet}).
  \item $d_L$ is bounded above by the metric: for all $x,y$ in $X$, $d_L(x,y)\le d(x,y)$. 
  \item $d_L$ is jointly lower-semicontinuous.
 \end{enumerate}
\end{prop}

\begin{proof}
 \begin{enumerate}
  \item If $x\le y$, then for all short monotone functions $f$, we have $f(x)- f(y)\le 0$. The supremum is attained by $f=0$. 
  
  \item Suppose that $X$ is L-ordered. If
  $$
  d_L(x,y) = \sup_{f:X\to\R} \left( f(x) - f(y) \right) = 0,
  $$
  for all short, monotone maps $f:X\to\R$, 
  $$
  f(x) - f(y) \le 0,
  $$
  which means $f(x)\le f(y)$. Since $X$ is L-ordered, then $x\le y$.
  
  Suppose now that $X$ is \emph{not} L-ordered. Then there exist $x\nleq y$ such that for all short monotone $f:X\to\R$, $f(x)\le f(y)$. But then
  \begin{align*}
   d_L(x,y) = \sup_{f:X\to\R} \left( f(x) - f(y) \right) \le 0 ,
  \end{align*}
  and again the supremum is attained by $f=0$.
  
  \item Let $x,y,z\in X$. Then
 \begin{align*}
  d_L (x,z) &= \sup_{f:X\to\R} \left( f(x) - f(z) \right) \\
   &= \sup_{f:X\to\R} \left( f(x) - f(y) + f(y) - f(z) \right) \\
   &\le \sup_{f:X\to\R} \left( f(x) - f(y) \right) + \sup_{f:X\to\R} \left( f(y) - f(z) \right) \\
   &= d_L(x,y) + d_L(y,z) .
 \end{align*}
  \item For all $x,y\in X$,
  \begin{align*}
   d_L(x,y) &= \sup \: \{ \: f(x) - f(y) \; : \; f\mbox{ short and monotone} \} \\
   & \le \sup \: \{  \: f(x) - f(y) \; : \; f\mbox{ short} \}  = d(x,y) .
  \end{align*}
  
  \item $d_L$ is defined as a pointwise supremum of continuous functions, therefore it is lower-semicontinuous. \qedhere
 \end{enumerate}
\end{proof}

We are now ready to prove the duality theorem.

\begin{proof}[Proof of Theorem~\ref{Ldual}]
 Suppose that for all short, monotone $f:X\to\R$,
 \begin{equation*}
  \int f \, dp \le \int f \, dq,
 \end{equation*}
 or in other words, 
 \begin{equation*}
  \sup_{f:X\to\R} \left(  \int f \, dp - \int f \, dq \right) = 0 ,
 \end{equation*}
 where the supremum is taken over all short, monotone maps.
 Now the short monotone maps are precisely those which satisfy $f(x) - f(y) \le d_L(x,y)$ for all $x$ and $y$; this inequality holds for all short monotone $f$ by the very definition of $d_L$, and if the inequality holds, then shortness and monotonicity are both easily implied. Since $d_L$ is lower-semicontinuous and satisfies the triangle inequality by Proposition~\ref{dlproperties}, we can apply Kantorovich duality in the form of Corollary~\ref{kantorovichduality2} to obtain
 \begin{align*}
  0 = \sup_{f:X\to\R} \left(  \int f \, dp - \int f \, dq \right) = \min_{r\in\Gamma(p,q)} \int_{X\otimes X} d_L(x,y) \, dr(x,y)
 \end{align*}
 where the minimizing $r$ exists (Corollary~\ref{mincouplingexists}). 
 In other words, there exists a coupling $r$ entirely supported on 
 $$\{d_L(x,y)=0\}.$$ 
 Since $X$ is L-ordered, all the points in the set above are contained in $\{\le\}$. 
 So $r$ is supported on $\{\le\}$, which means that $p\le q$. 
 
 Conversely, suppose that such a coupling $r$ exists. Then we can use the following standard  Kantorovich duality argument,
 \begin{equation*}
   \int f \, dp - \int f \, dq = \int_{X\times X} \big( f(x) - f(y)  \big) \, dr \le 0.\,
 \end{equation*}
 as $r$ is supported on a region where $f(x)\le f(y)$. 
\end{proof}

From this characterization, it is easy to see that the order on $PX$ is closed and transitive. Antisymmetry will be proven in Corollary~\ref{antisymmetry}.

\begin{cor}\label{PpreservesLorder}
 Let $X$ be an L-ordered metric space. Then $PX$ is L-ordered too. 
\end{cor}

\begin{proof}
 Given a short, monotone map $f:X\to\R$, the assignment 
 $$
 p\mapsto \int f \, dp
 $$
 is short and monotone as a map $PX\to\R$.
 By Theorem~\ref{Ldual}, this determines the order. Therefore $PX$ is L-ordered.
\end{proof}

\subsection{Antisymmetry of the stochastic order}\label{ssecantis}

Here we prove that the stochastic order on any L-ordered space is a partial order, i.e.~it is antisymmetric. 
It has been an open question whether antisymmetry holds on every ordered metric space. This is known to be true for compact spaces~\cite{edwards}, and for certain types of cones in Banach spaces~\cite[Theorem~4.3]{hll}. After the current work was completed, the first named author has shown that the stochastic order is antisymmetric for Radon measures on \emph{any} topological space~\cite{antisymmetry}.

For the very large class of L-ordered spaces, we can prove antisymmetry using a Kantorovich duality argument, encoded in the following statement.

\begin{prop}\label{preantisymmetry}
 Let $X$ be an L-ordered metric space. Let $p,q\in PX$, and suppose that $p<q$ strictly. Then there exists a short monotone $f:X\to\R$ such that 
 \begin{align*}
  \int f \, dq > \int f \, dp \quad\mbox{strictly.}
 \end{align*}
\end{prop}

\begin{proof}
 Our assumption is that $p\le q$ and $p\ne q$. Then there exists a coupling $r$ supported on the relation $\{\le\}$, which cannot be supported only on the diagonal $D:=\{(x,x)\}$, because we would otherwise have $p = q$ since the two projection maps $X \otimes X \to X$ are equal on the diagonal. In other words, there exists a point $(\bar{x},\bar{y})$ with $\bar{x} < \bar{y}$ strictly, and every open neighborhood of $(\bar{x},\bar{y})$ has strictly positive $r$-measure. Since $X$ is L-ordered and $\bar{y}\nleq\bar{x}$, there exists a short, monotone map $f:X\to\R$ such that $f(\bar{y})>f(\bar{x})$ strictly. We can then choose an open neighborhood $U$ of $(\bar{x},\bar{y})$ which is disjoint from the diagonal, and on which the function 
 $$
 (x,y) \mapsto f(y)-f(x)
 $$
 is strictly positive. 
 Therefore, 
 \begin{align*}
  \int f \, dq - \int f \, dp &= \int_{X\otimes X} \big( f(y)-f(x) \big) \, dr(x,y) \\
  &\ge \int_{U}  \big( f(y)-f(x) \big) \, dr(x,y) > 0
 \end{align*}
 strictly, which in turn means that
 \begin{align*}
  \int f \, dq > \int f \, dp .
 \end{align*}
\end{proof}

\begin{cor}\label{antisymmetry}
 Let $X$ be an L-ordered metric space. Then the stochastic order on $PX$ is antisymmetric.
\end{cor}

\begin{proof}
 Let $p,q\in PX$, and suppose that both $p\le q$ and $q\le p$ in the stochastic order.
 Then necessarily
 $$
 \int f \, dp \, = \int f \, dq 
 $$
 for all short monotone maps $f:X\to\R$. By Proposition~\ref{preantisymmetry}, then, it must be that $p=q$.
\end{proof}

Since all ordered Banach spaces are L-ordered (see Corollary~\ref{oBan_Lorder}), and therefore so are all subsets of Banach spaces, this is a broad generalization of~\cite[Theorem~4.3]{hll}.

Edwards observes~\cite[p.~71]{edwards} that antisymmetry holds on all \emph{completely regular ordered spaces}, as defined in~\cite[p.52--54]{nachbin}. This does not imply our result: not all L-ordered spaces are completely regular ordered, as the following counterexample shows.

\begin{eg}
 Let $X$ be the Banach space $\ell^{\infty}$ of bounded sequences $\N\to\R$ equipped with the norm:
 $$
 \big\| \{x_i\}_{i=0}^\infty \big\| := \sup_{i} |x_i| .
 $$
 Consider the following cone (this construction was given to us by Rostislav Matveev):
 $$
 \left\{\{x_i\}_{i=0}^\infty \; \big| \; \forall i \ge 1,\, x_0 \ge \dfrac{1}{i}\,|x_i| \right\} .
 $$
 This cone is norm-closed, it makes $X$ an ordered Banach space, so in particular it is L-ordered (see Corollary~\ref{oBan_Lorder}). However, for increasing $i$, the cone has larger and larger aperture (so that it is not \emph{normal} in the sense of Krein~\cite[Definition~2.18]{conesduality}).
 Consider now the order interval between $(-1,0,\dots,0,\dots)$ and $(+1,0,\dots,0,\dots)$. This is expressed by the condition that $-1\le x_0\le 1$, and that for every $i \ge 1$,
 $$
  x_0 -1 \le \dfrac{1}{i}\,|x_i| \le x_0 +1 .
 $$
 Fix now an $i\in N$ and consider the sequence $\{x_i\}$ which is zero except at position $i$, where $x_i=i/2$. This is contained in the order interval considered above, since the constraint at $x_i$ reads
 $$
  -1 \le \dfrac{1}{2} \le +1 .
 $$
 However, the norm of this sequence is $i/2$, and $i$ can be chosen arbitrarily large. Therefore this order interval is not norm-bounded. 
 Every intersection of a lower and an upper open set in $X$ contains an order interval of the form above, possibly translated and rescaled, so that any such intersection would also not be norm-bounded. It follows that intersections of upper and lower open sets do not form a neighborhood basis. So, in particular, $X$ cannot be completely regular ordered. 
\end{eg}

\section{The ordered Kantorovich monad}
\label{secordp}

In~\cite{ours_kantorovich}, we showed that in the unordered case, the functor $P$ carries a monad structure whose algebras are the closed convex subsets of Banach spaces. Here we show that this monad structure can be lifted to the category $\lcomet$. The easiest way to do this is to show that all the structure maps are monotone between the respective orders, so that the commutativity of the necessary diagrams is inherited from $\cat{CMet}$. This will be done in~\ref{ssecmonadstr}. In~\ref{ssecordms}, we show that $P$ is also a bimonoidal monad as in the unordered case~\cite[Section~5]{ours_bimonoidal}. In Section~\ref{secordalg}, we will study the category of $P$-algebras and prove a number of general properties, including a characterization of $P$-algebras as closed convex subsets of ordered Banach spaces.

\subsection{Monad structure}\label{ssecmonadstr}

First of all, by Corollary~\ref{PpreservesLorder}, if $X\in\lcomet$, then $PX\in\lcomet$ too.

We will now lift the Kantorovich monad of \cite{ours_kantorovich} to $\lcomet$. To do this, we have to:
\begin{enumerate}
 \item Show that if $f : X \to Y$ is monotone, then also $Pf : PX \to PY$ is monotone.
 \item Show that the structure transformations have components $\delta : X \to PX$ and $E : PPX \to PX$ which are monotone. 
\end{enumerate}

The commutativity of all relevant diagrams involved is obvious, since the forgetful functor $\lcomet \to \cat{CMet}$ is faithful.

We start with the first item. As in~\cite{ours_kantorovich}, we also write $f_*$ as shorthand for $Pf$, since $Pf$ takes a probability measure on $X$ to its pushforward on $Y$.

\begin{prop}\label{Pfmonotone}
 Let $f:X\to Y$ be short and monotone. Then $Pf:PX\to PY$ is also monotone. 
\end{prop}

\begin{proof}
 Let $p\leq q$ in $PX$. Then we have to prove that for every closed upper set $C\subseteq Y$,
 \begin{align*}
  (f_*p)(C) \le (f_*q)(C) ,
 \end{align*}
 which means  
 \begin{align*}
  p(f^{-1}(C)) \le q(f^{-1}(C)) .
 \end{align*}
 Now since $f$ is continuous, $f^{-1}(C)$ is closed. Since $f$ is monotone, $f^{-1}(C)$ is an upper set. By definition of the order on $PX$, $p(C')\le q(C')$ for all upper closed sets $C'$.
 Therefore $(f_*p)(C) \le (f_*q)(C)$.
\end{proof}

Hence $P$ is indeed an endofunctor of $\lcomet$. Since we have not actually used L-orderedness, this still works with $\cat{COMet}$ in place of $\lcomet$. But since we want to focus on L-ordered spaces especially for the purposes of Section~\ref{secordalg}, we nevertheless formulate the statements themselves for $\lcomet$.
To prove the monotonicity of the structure maps, in particular, we will use the dual characterization of the order in terms of monotone short maps of Theorem~\ref{Ldual} (this approach only works in $\lcomet$).

\begin{prop}
\label{Pmonadmonotone}
 Let $X\in\lcomet$. Then;
 \begin{enumerate}
  \item $\delta:X\to PX$ is an order embedding;
  \item $E:PPX\to PX$ is monotone. 
 \end{enumerate}
\end{prop}

\begin{proof} 
\begin{enumerate}
 \item Let $x\le y\in X$, and let $f:X\to\R$ (short, monotone). Then
 \begin{align*}
  \int_{X} f\,d \delta(x) = f(x) \le f(y) = \int_{X} f\,d \delta(y).
 \end{align*}
 Therefore $\delta(x)\le\delta(y)$. The converse follows similarly, using Theorem~\ref{Ldual}.
 
 \item Let $\mu\le\nu$ in $PPX$, and again $f:X\to\R$ short and monotone. By Theorem~\ref{Ldual}, the assignment
 \begin{equation*}
  p \longmapsto \int_X f \,dp
 \end{equation*}
 is monotone as a function $PX\to\R$. Therefore we can write
 \begin{align*}
  \int_{X} f\,d(E\mu) = \int_{PX} \left( \int_X f\,dp \right) d\mu(p) \le \int_{PX} \left( \int_X f\,dp \right) d\nu(p) = \int_{X} f\,d(E\nu).
 \end{align*}
 By Theorem~\ref{Ldual}, we conclude that $E\mu\le E\nu$, so that $E$ is indeed monotone. \qedhere
\end{enumerate}
\end{proof}

We therefore obtain:

\begin{cor}
 $(P,\delta,E)$ is a monad on $\lcomet$ lifting the Kantorovich monad on $\cat{CMet}$. 
\end{cor}

We will call this monad with the same name whenever this should not cause confusion.

\subsection{Bimonoidal structure}\label{ssecordms}

The monad $P$ on $\cat{CMet}$ has a bimonoidal structure, which corresponds to the formation of product and marginal distributions~\cite{ours_bimonoidal}. We now extend this structure to $\lcomet$. 

First of all, is easy to see that if $X$ and $Y$ are $L$-ordered, then $X\otimes Y$ is as well.
Just like for the monad structure, it now suffices to show that its structure maps given by the formation of product distributions $\nabla:PX\otimes PY \to P(X\otimes Y)$ and the formation of marginals $\Delta:P(X\otimes Y) \to PX\otimes PY$ are monotone. 

\begin{lemma}\label{nablasm}
 Let $X,Y\in \lcomet$. Then $\nabla:PX\otimes PY \to P(X\otimes Y)$ is monotone. 
\end{lemma}

\begin{proof}
 First of all, let $f:X\otimes Y\to\R$ be monotone, and let $p\in PX$. Then the function
 \begin{equation}
 \left(  \int_X f(x,-)\,dp(x) \right) : Y\to \R
 \end{equation}
 is monotone as well.

 Suppose now that $p\le p'$ in $PX$ and $q\le q'$ in $PY$. Let $f:X\otimes Y\to \R$ be monotone. Then using the remark above and its counterpart for $X \to \R$,
 \begin{align*}
  \int_{X\otimes Y} f(x,y)\,d(p\otimes q)(x,y) &= \int_Y \left( \int_X f(x,y)\,dp(x) \right)\,dq(y) \\
  &\le \int_Y \left( \int_X f(x,y)\,dp(x) \right)\,dq'(y) \\
  &=  \int_X \left( \int_Y f(x,y)\,dq'(y) \right)\,dp(x) \\
  &\le \int_X \left( \int_Y f(x,y)\,dq'(y) \right)\,dp'(x) \\
  &= \int_{X\otimes Y} f(x,y)\,d(p'\otimes q')(x,y) .
 \end{align*} \qedhere
\end{proof}

\begin{lemma}\label{deltasm}
 Let $X,Y\in \lcomet$. Then $\Delta:P(X\otimes Y) \to PX\otimes PY$ is monotone.
\end{lemma}

\begin{proof}
Suppose that $p\le q$ in $P(X\otimes Y)$. In order to prove the claim, we then need to show that $p_X \le q_X$ for the marginal on $X$, and similarly for $Y$, which works analogously.

So let $f:X\to\R$ be monotone. Upon composing with the projection, it is also monotone as a function $X\otimes Y\to\R$.
This means that
\begin{align*}
 \int_{X\otimes Y} f(x) \,dp(x,y) \le \int_{X\otimes Y} f(x) \,dq(x,y) ,
\end{align*}
but we can rewrite both terms as
\begin{align*}
 \int_{X} f(x) \,dp_X(x) \le \int_{X} f(x) \,dq_X(x) ,
\end{align*}
and therefore $p_X\le q_X$. 
\end{proof}

Together with the results of \cite[Section 5]{ours_bimonoidal}, we get as a corollary:

\begin{thm}\label{bimonord}
 $P$ is a symmetric bimonoidal monad on $\lcomet$.
\end{thm}

This in particular implies (see~\cite[Proposition~4.1]{ours_bimonoidal}) that $\Delta\circ\nabla=\id$, and therefore $\nabla$, as the formation of product distributions, is an order embedding in addition to being a metric embedding.

\section{Algebras of the ordered Kantorovich monad}\label{secordalg}

\subsection{Algebras as subsets of ordered Banach spaces}

We previously showed that the category of algebras of the Kantorovich monad on $\cat{CMet}$ is equivalent to the category of closed convex subsets of Banach spaces \cite[Section 5]{ours_kantorovich}. The algebra map $e : PA \to A$ maps every probability measure to its barycenter,
$$
p \longmapsto \int a \, dp(a)
$$
and the morphisms of algebras are the short affine maps, i.e.~the maps which commute with integration. In this section, we extend this result to the category of algebras of the ordered Kantorovich monad $P$ on $\lcomet$, by showing that it is equivalent to the category of closed convex subsets of \emph{ordered} Banach spaces, for which the algebra map $e : PA \to A$ is monotone; and the morphisms of algebras are then just the \emph{monotone} short affine maps. This is analogous to what happens with ordered algebras of the Radon monad~\cite{keimel}.

The arguments required to prove this result are surprisingly complicated, and we will see that the L-orderedness is a necessary and sufficient condition for an embedding into an ordered Banach space to exist.

\begin{lemma}\label{justbinary}
Let $A\in\lcomet$ be an algebra of the unordered Kantorovich monad via an algebra map $e : PA \to A$. Then $e$ is monotone if and only if for all $a,b,c\in A$ and $\lambda\in [0,1]$,
\[
 a \le b \quad\Rightarrow\quad e(\lambda \delta_a + (1-\lambda) \delta_c) \;\le\; e(\lambda \delta_b + (1-\lambda) \delta_c).
\]
\end{lemma}

This result is the ordered analogue of the equivalence of (a) and (d) in~\cite[Theorem~5.2.1]{ours_kantorovich}. The condition is the defining property of an \emph{ordered barycentric algebra} \cite{keimel}. 

\begin{proof}
The assumption $a \le b$ implies $\lambda \delta_a + (1-\lambda) \delta_c \le \lambda \delta_b + (1-\lambda) \delta_c$ in $PA$. So if $e$ is monotone, then the conclusion follows.

Conversely, suppose that the above implication holds. In order to prove that $e$ is monotone, the density result of Proposition~\ref{orderdensity} shows that it is enough to prove $e(i_N(x_n)) \le e(i_N(y_n))$ for $(x_n),(y_n)\in X^N$ with $i(x_n) \le i(y_n)$. By Corollary~\ref{splittingcor}, we can relabel $(y_n)$ by a permutation such that $x_n \leq y_n$ for every $n\in N$. Writing $N = \{1,\ldots,|N|\}$, we therefore have
\[
	e\left( \frac{1}{|N|} \sum_{i=1}^{k  } \delta_{x_i} + \frac{1}{|N|} \sum_{i=k+1}^{|N|} \delta_{y_i} \right) \le
	e\left( \frac{1}{|N|} \sum_{i=1}^{k-1} \delta_{x_i} + \frac{1}{|N|} \sum_{i=k  }^{|N|} \delta_{y_i} \right)
\]
as an instance of the assumption, for every $k=1,\ldots,|N|$. Chaining all these inequalities results in the claimed $e(i_N(x_n)) \le e(i_N(y_n))$.
\end{proof}

So if we represent $A$ as a closed convex subset of a Banach space, then $e$ is monotone if and only if
\begin{equation}
 a \le b \quad\Rightarrow\quad \lambda\,a + (1-\lambda)\,c \;\le\; \lambda\,b + (1-\lambda)\,c
\end{equation}
holds for all $a,b,c\in A$ and $\lambda\in[0,1]$, since the right-hand side is exactly $e(\lambda \delta_a + (1-\lambda) \delta_b) \leq e(\lambda \delta_b + (1-\lambda) \delta_c)$. 

We will prove in~\ref{strictlymonotone} that when the map $e$ is monotone, then it is even \emph{strictly} monotone. 

\begin{deph}
 An \emph{ordered Banach space} is a Banach space equipped with a closed convex cone called the \emph{positive cone}.
\end{deph}

As usual, two elements of the Banach space are ordered if and only if their difference is in the positive cone.

We already know that every closed convex subset $A$ of an ordered Banach space is a $P$-algebra in $\cat{CMet}$, with the structure map given by integration, and we know that integration is monotone thanks to Lemma~\ref{justbinary}. So in order for $A$ to be a $P$-algebra in $\lcomet$, what remains to be checked is that $A$ is indeed an object of $\lcomet$, i.e.~that it is L-ordered. This is guaranteed by the Hahn-Banach theorem, which even shows that we can test the order using only \emph{linear} short monotone maps:

\begin{prop}\label{applyhahnbanach}
 Let $B$ be an ordered Banach space. Let $a,b\in B$. Then $a\le b$ if and only if for every short monotone \emph{linear} functional $h:B\to\R$, $h(a)\le h(b)$. 
\end{prop}

\begin{proof}
The ``only if'' part is trivial by monotonicity of $h$, so we focus on the ``if'' direction. 

 We write $B^+$ for the positive cone of $B$. Suppose that $a\nleq b$.  
 This means that the point $v:= b - a$ does \emph{not} lie in $B^+$. Since $\{v\}\subseteq B$ is trivially compact and $B^+\subseteq B$ is closed and convex, the Hahn-Banach separation gives us a bounded linear functional $h: B\to\R$ such that
 \begin{enumerate}
  \item\label{hmonotone} $h(c) \ge 0$ for all $c\in B^+$, and
  \item\label{hseparates} $h(v)< 0$.
 \end{enumerate}
 Without loss of generality, we can rescale $h$ to norm one, so that it is short. Property~\ref{hmonotone} means exactly that $h$ is monotone.
 By linearity, property~\ref{hseparates} means exactly that $h(a) > h(b)$, which is enough.
\end{proof}

\begin{cor}
\label{oBan_Lorder}
 Every ordered Banach space is L-ordered. 
\end{cor}

\begin{remark}\label{lordembed}
	Note that $L$-orderedness is preserved by passing to a subspace. Indeed, given a subset $S$ of an $L$-ordered space $X$, we can test the order on $S$ using in particular the restrictions to $S$ of the short, monotone maps on $X$.
	
	As a special case, subsets of ordered Banach spaces are $L$-ordered. In fact, e.g.~by applying the upcoming Theorem~\ref{poscone} to $PX$ and composing the embedding with the embedding $\delta : X \to PX$, it follows that the class of L-ordered metric spaces \emph{coincides} with the class of subsets of ordered Banach spaces up to isomorphism. 
\end{remark}

Getting back to $P$-algebras, we have now shown:

\begin{cor}
 Every closed convex subset of an ordered Banach space is a $P$-algebra in $\lcomet$. 
\end{cor}

Here is the converse statement:

\begin{thm}\label{poscone}
 Every $P$-algebra in $\lcomet$ is isomorphic to a closed convex subset of an ordered Banach space.
\end{thm}

The proof follows that of the analogous result for ordered barycentric algebras~\cite[Proposition 3.3]{keimel}. We also need the following technical result about the L-distance on a $P$-algebra.

\begin{lemma}\label{Ldistbound}
 Let $A$ be a $P$-algebra. Let $x,y,z\in A$ and $\alpha\in [0,1]$. Then
 \begin{equation}
  d_L\big( \alpha\,x + (1-\alpha)\, z, \alpha\,y + (1-\alpha)\, z \big) = \alpha\, d_L(x,y) .
 \end{equation}
\end{lemma}

\begin{proof}[Proof of Lemma~\ref{Ldistbound}]
 The proof works along the lines of~\cite[Lemma~8]{capraro}.
 We know by Remark~\ref{smL} that since $e$ is short and monotone, the composition
 $$
 x \mapsto \alpha\,x + (1-\alpha)\,z = e( \alpha\,\delta_x + (1-\alpha)\,\delta_z) 
 $$
 is short and monotone as well, so that we have
 \begin{equation}
 \label{dL_decreasing}
  d_L\big( \alpha\,x + (1-\alpha)\, z, \alpha\,y + (1-\alpha)\, z \big) \le \alpha\, d_L(x,y) .
 \end{equation}
 Now by setting $z=y$, we get that 
 \begin{equation*}
  d_L\big( \alpha\,x + (1-\alpha)\, y, y \big) \le \alpha\, d_L(x,y) ,
 \end{equation*}
 and by setting instead $z=x$, we get 
 \begin{equation*}
  d_L\big( x, \alpha\,y + (1-\alpha)\, x \big) \le \alpha\, d_L(x,y) .
 \end{equation*}
 By swapping $\alpha$ and $(1-\alpha)$ we also get 
 \begin{equation*}
  d_L\big( x, \alpha\,x + (1-\alpha)\, y \big) \le (1-\alpha)\, d_L(x,y) ,
 \end{equation*}
 which we will use below.
 Now, by the triangle inequality,
 \begin{align*}
  d_L(x,y) &\le d_L\big(x, \alpha\,x + (1-\alpha)\, y\big) +  d_L\big( \alpha\,x + (1-\alpha)\, y, y \big) \\
  &\le d_L\big(x, \alpha\,x + (1-\alpha)\, y\big) +  \alpha\,d_L\big( x , y \big) \\
  &\le (1-\alpha)\, d_L\big(x , y\big) +  \alpha\,d_L\big( x , y \big) = d_L(x,y),
 \end{align*}
 so that all three inequalities must be equalities. In particular, $d_L\big( \alpha\,x + (1-\alpha)\, y, y \big) = \alpha\,d_L\big( x , y \big)$, and $d_L\big(x, \alpha\,x + (1-\alpha)\, y\big) =  (1-\alpha)\, d_L\big(x , y\big)$. The function $d_L$ is therefore ``affine on lines'', or ``longitudinally translation-invariant''.

Denoting $\alpha\,x + (1-\alpha)\, z$ by $x_\alpha$ and $\alpha\,y + (1-\alpha)\, z$ by $y_\alpha$, we obtain the situation
 
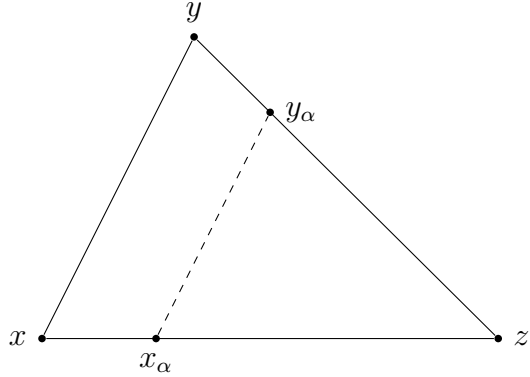
\begin{center}
 \begin{tikzpicture}[baseline=(current  bounding  box.center),>=stealth]
  \node[bullet,label=left:$x$] (x) at (0,0) {};
  \node[bullet,label=above:$y$] (y) at (2,4) {};
  \node[bullet,label=right:$z$] (z) at (6,0) {};
  \draw[-] (x) -- (y);
  \draw[-] (y) -- (z);
  \draw[-] (x) -- (z);
  \node[bullet,label=below:$x_\alpha$] (xa) at (1.5,0) {};
  \node[bullet,label=right:$y_\alpha$] (ya) at (3,3) {};
  \draw[-,dashed] (xa) -- (ya);
 \end{tikzpicture}
 \end{center} 
 What we have shown above is $d_L(x_\alpha,z) = \alpha\,d_L(x,z)$, and likewise $d_L(y_\alpha,z) = \alpha\,d_L(y,z)$. We have to prove that $d_L(x_\alpha,y_\alpha)=\alpha\,d_L(x,y)$. Consider now the point $y':=\alpha\,y + (1-\alpha)\, x$, which forms a parallelogram with $x,x_\alpha$ and $y_\alpha$:
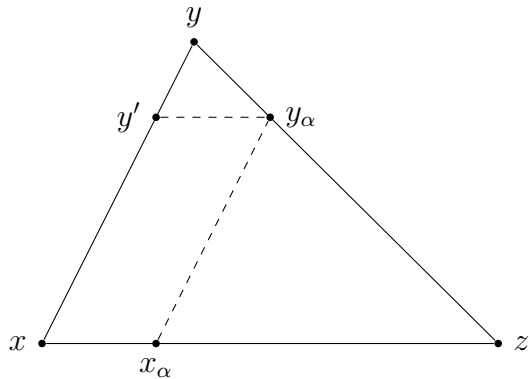
 \begin{center}
 \begin{tikzpicture}[baseline=(current  bounding  box.center),>=stealth]
  \node[bullet,label=left:$x$] (x) at (0,0) {};
  \node[bullet,label=above:$y$] (y) at (2,4) {};
  \node[bullet,label=right:$z$] (z) at (6,0) {};
  \draw[-] (x) -- (y);
  \draw[-] (y) -- (z);
  \draw[-] (x) -- (z);
  \node[bullet,label=below:$x_\alpha$] (xa) at (1.5,0) {};
  \node[bullet,label=right:$y_\alpha$] (ya) at (3,3) {};
  \draw[-,dashed] (xa) -- (ya);
  \node[bullet,label=left:$y'$] (y') at (1.5,3) {};
  \draw[-,dashed] (y') -- (ya);
 \end{tikzpicture}
 \end{center} 
 If we proved that $d_L$ is translation invariant in the sense that $d_L(x_\alpha,y_\alpha) = d_L(x,y')$, then we would conclude that $d_L(x_\alpha,y_\alpha) = d_L(x,y')=\alpha\,d_L(x,y)$, thereby proving the claim.
 
 Now for $\e\in(0,1)$, consider the points 
 $$
 x_{\e} := \e\,x_\alpha + (1-\e)\, x , \quad y_{\e} := \e\,y_\alpha + (1-\e)\, y',
 $$
 and 
 $$
 k_{\e} := (1-\e)\,x_{\e} + \e\,y_{\e} = \e\,y_\alpha + (1-\e)\,x,
 $$
 which can be illustrated as
 
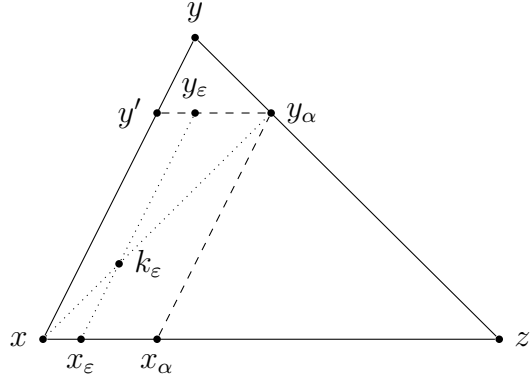
\begin{center}
 \begin{tikzpicture}[baseline=(current  bounding  box.center),>=stealth]
  \node[bullet,label=left:$x$] (x) at (0,0) {};
  \node[bullet,label=above:$y$] (y) at (2,4) {};
  \node[bullet,label=right:$z$] (z) at (6,0) {};
  \draw[-] (x) -- (y);
  \draw[-] (y) -- (z);
  \draw[-] (x) -- (z);
  \node[bullet,label=below:$x_\alpha$] (xa) at (1.5,0) {};
  \node[bullet,label=right:$y_\alpha$] (ya) at (3,3) {};
  \draw[-,dashed] (xa) -- (ya);
  \node[bullet,label=left:$y'$] (y') at (1.5,3) {};
  \draw[-,dashed] (y') -- (ya);
  \draw[-,dotted] (x) -- (ya);
  \node[bullet,label=above:$y_{\e}$] (xe) at (2,3) {};
  \node[bullet,label=below:$x_{\e}$] (ye) at (0.5,0) {};
  \draw[-,dotted] (xe) -- (ye);
  \node[bullet,label=right:$k_{\e}$] (ke) at (1,1) {};
 \end{tikzpicture}
 \end{center} 
 By~\eqref{dL_decreasing}, we have
 $$
 d_L(x_{\e},k_{\e}) = d_L\big( \e\,x_\alpha + (1-\e)\, x, \e\,y_\alpha + (1-\e)\,x \big) \le \e \,d_L(x_\alpha,y_\alpha) .
 $$
 Moreover, since $k_{\e}$ is on the line connecting $x_{\e}$ and $y_{\e}$, 
 $$
 d_L(x_{\e},k_{\e}) = \e\,d_L(x_{\e},y_{\e}). 
 $$
 Therefore, 
 $$
 d_L(x_{\e},y_{\e}) = \e^{-1}\, d_L(x_{\e},k_{\e}) \le  d_L(x_\alpha,y_\alpha).
 $$
 We now take the limit $\e\to 0$.
 Note that, since $d_L$ is lower semicontinuous, 
 $$
 d_L(x,y')\le d_L(x_{\e},y_{\e}) \le d_L(x_\alpha,y_\alpha).
 $$
 Analogously,  
 $$
 d_L(k_{\e},y_{\e}) = (1-\e) \, d_L(x_{\e},y_{\e}) ,
 $$
 therefore
 $$
 d_L(x_{\e},y_{\e}) = (1-\e)^{-1}\, d_L(k_{\e},y_{\e}) \le  d_L(x,y').
 $$
 We can now let $\e\to 1$, and again by lower semicontinuity,
 $$
 d_L(x_\alpha,y_\alpha) \le d_L(x_{\e},y_{\e})  \le d_L(x,y').
 $$
 So $d_L(x,y') = d_L(x_\alpha,y_\alpha)$, as was to be shown.
\end{proof}

We can now prove the theorem.

\begin{proof}[Proof of Theorem~\ref{poscone}]
By what we already know, it is enough to show that if $B$ is a Banach space and $A\subseteq B$ is a closed convex subset equipped with a an L-order, then we can equip $B$ itself with a closed partial order that restricts to the given order on $A$. So let $x\in B$ be considered positive if it is of the form $\lambda(y_+ - y_-)$ for $\lambda\geq 0$ and $y_+\geq y_-$ in $A$. Using the fact that taking convex combinations in $A$ is monotone, it is easy to see that this defines a convex cone. Taking $x\geq y$ if and only if $x-y$ is in the cone recovers the original order, since $x-y = \lambda(z_+ - z_-)$ for $z_+\geq z_-$ in $A$ and $\lambda > 0$ implies $\tfrac{1}{1+\lambda}x + \tfrac{\lambda}{1+\lambda}z_- = \tfrac{1}{1+\lambda}y + \tfrac{\lambda}{1+\lambda}z_+$. Together with $z_+\geq z_-$, we hence obtain $x\leq y$ from the general theory of ordered topological barycentric algebras~\cite[Corollary~4.2]{keimel}.

We cannot assume that the cone in $B$ defined this way is closed, so we take its closure. To check that the resulting embedding is still an order embedding, we have to show that the order of $A$ already contains all the order relations that are added by taking the closure of the cone.
In other words, we have to prove that whenever the sequence $\lambda_n (z_{+n} - z_{-n})$
for some $\lambda_n\ge 0$ and $z_{+n}\ge z_{-n}\in A$ tends to $y-x$, then already $x\le y$ in the order of $A$. So suppose that
$$
d\left( \lambda_n (z_{+n} - z_{-n}),  y-x \right) \to 0,
$$
or, rewriting everything in terms of only convex combinations (elements of $A$), using $\alpha_n := \frac{1}{1+\lambda_n}$,
\begin{equation}\label{distlimit}
 \dfrac{1}{\alpha_n} \, d\big( \alpha_n\,x + (1-\alpha_n)\,z_{+n}, \alpha_n\,y + (1-\alpha_n)\,z_{-n} \big) \to 0,
\end{equation}
Now consider the L-distance on $A$. We have from Lemma~\ref{Ldistbound} and the triangle inequality for $d_L$ that
\begin{align*}
 d_L(x,y) &= \dfrac{1}{\alpha_n} \, d_L\big( \alpha_n\,x + (1-\alpha_n)\,z_{+n}, \alpha_n\,y + (1-\alpha_n)\,z_{+n} \big) \\
 &\le \dfrac{1}{\alpha_n} \, d_L\big( \alpha_n\,x + (1-\alpha_n)\,z_{+n}, \alpha_n\,y + (1-\alpha_n)\,z_{-n} \big) \\
 &\quad + \dfrac{1}{\alpha_n} \, d_L\big( \alpha_n\,y + (1-\alpha_n)\,z_{-n}, \alpha_n\,y + (1-\alpha_n)\,z_{+n} \big) \\
 &= \dfrac{1}{\alpha_n} \, d_L\big( \alpha_n\,x + (1-\alpha_n)\,z_{+n}, \alpha_n\,y + (1-\alpha_n)\,z_{-n} \big) \\
 &\quad\quad\quad + \dfrac{1-\alpha_n}{\alpha_n} \, d_L\big( z_{-n}, z_{+n} \big) \\
 &= \dfrac{1}{\alpha_n} \, d_L\big( \alpha_n\,x + (1-\alpha_n)\,z_{+n}, \alpha_n\,y + (1-\alpha_n)\,z_{-n} \big) + 0,
\end{align*}
since $z_{-n}\le z_{+n}$. Since the L-distance on $A$ is bounded above by the usual distance, the expression above is bounded by the quantity~\eqref{distlimit}, which by assumption tends to zero, so necessarily $d_L(x,y)=0$. Since $A$ is L-ordered, then by Proposition~\ref{dlproperties} we have that $x\le y$. 
\end{proof}

In the unordered case, the morphisms of $P$-algebras are the short affine maps, i.e.\ the short maps which respect convex combinations. In the ordered case, they are additionally required to be monotone. Overall, we therefore have:

\begin{thm}
For $P$ the ordered Kantorovich monad on $\lcomet$, the category of $P$-algebras is equivalent to the category of closed convex subsets of ordered Banach spaces with short affine monotone maps.
\end{thm}

We will refer to $P$-algebras in $\lcomet$ as \emph{ordered $P$-algebras}. These of course include those with trivial order. 

As it is well-known, any monad induces an adjunction with its category of algebras~\cite[VI.2, Theorem~1]{maclane}. Therefore we have a natural bijection
\begin{equation}\label{choqadj2}
\lcomet(X,A)\cong \lcomet^P(PX,A),
\end{equation} 
where $A$ on the right is considered as a $P$-algebra, and on the left as its underlying $L$-ordered space.
We can interpret this as an ordered, noncompact version of Choquet theory~\cite[Chapter~1]{winkler}.
More concretely, this means the following. A short, monotone, map $f:X\to A$ from a complete ordered space $X$ to an ordered convex space ($P$-algebra) $A$ is uniquely determined by the affine extension it defines, as an affine map on probability measures
\begin{equation*}
 p\longmapsto \int_X f \,dp ,
\end{equation*}
i.e.\ the $P$-morphism given by the composition
\begin{equation}\label{integral}
 \begin{tikzcd}
  PX \ar{r}{Pf} & PA \ar{r}{e} & A ,
 \end{tikzcd}
\end{equation}
and every affine map $\tilde{f} : PX\to A$ can be written uniquely in this form, as the affine extension of a map $f:X\to A$. This affine is given by the restriction of $\tilde{f}$ to the extreme points of the simplex: $f$ is the composite
\begin{equation}\label{restriction}
 \begin{tikzcd}
  X \ar{r}{\delta} & PX \ar{r}{\tilde{f}} & A ,
 \end{tikzcd}
\end{equation}
and the original $\tilde{f}$ is its affine extension.
The standard theory of monads tells us that the operations $f\mapsto e\circ(Pf)$ and $\tilde{f}\mapsto \tilde{f}\circ\delta$ are inverse to each other, forming the natural bijection~\eqref{choqadj2}.

\subsection{The integration map is strictly monotone}\label{strictlymonotone}

For real random variables, it is well-known that if $p<q$ strictly, then $e(p)<e(q)$ strictly~\cite[Theorem~1]{fishburn}. The interpretation is that \emph{if one moves a nonzero amount of mass upwards in the order, then the center of mass will strictly rise}. Here we give a general version of the same statement, which applies to any ordered $P$-algebra, or equivalently to any closed convex subset of an ordered Banach space.

\begin{prop}\label{epev}
 Let $A$ be an ordered $P$-algebra, and let $p,q\in PA$. Suppose that $p\le q$ in the usual stochastic order, and $e(p)=e(q)$. Then $p=q$. 
\end{prop}

The proof is reminiscent of the proof of Proposition~\ref{preantisymmetry}.

\begin{proof}
By definition of the stochastic order, we know that there exists a joint $r\in P(A\otimes A)$ of $p$ and $q$ whose support lies entirely in the relation $\{\le\}\subset A\otimes A$. We want to prove that in fact, $r$ must be supported on the diagonal $D:=\{(a,a), a\in A\}$, since this implies that $p=q$.

We use an isometric order embedding $A \subseteq B$ into an ordered Banach space $B$, which we know to exist by Theorem~\ref{poscone}, and work with the pushforwards of $p$, $q$ and $r$ to $B$ instead. This way, we can assume $A = B$ without loss of generality, which we do from now on.

Now suppose that $r$ is \emph{not} entirely supported on the diagonal. Then there exists an $(a,b)\in B\otimes B$ with $a<b$ strictly, such that every open neighborhood of $(a,b)$ has strictly positive $r$-measure. The Hahn-Banach separation theorem (via Proposition~\ref{applyhahnbanach}) gives us a map $h : B \to \R$ which is short, linear, and monotone, and such that $h(a) < h(b)$. Now consider the integral
 \begin{equation}\label{inth}
  \int_{B\otimes B} (h(x) - h(y)) \, dr(x,y) .
 \end{equation}
 We have on the one hand, using that $h$ is linear,
 \begin{align*}
  &\int_{B\otimes B} (h(x) - h(y)) \, dr(x,y) \\
  &= \int_{B\otimes B} h(x) \, dr(x,y) - \int_{B\otimes B} h(y) \, dr(x,y) \\
  &= \int_B h(x) \, dp(x)  - \int_B h(y) \, dq(y) \\
  &= h\left( \int_B x \, dp(x) - \int_B y \, dq(y) \right) \\
  &= h\left( e(p) - e(q) \right) = 0.
 \end{align*}
 At the same time, we have that the integrand of~\eqref{inth} is continuous and nonnegative on the support of the measure $r$, while being strictly positive on $(a,b) \in \supp(r)$. This implies that the integral itself is strictly positive, a contradiction. Therefore our assumption that $r$ is not supported on $D$ must have been false.
\end{proof}

\subsection{2-Categorical structure}\label{ssechigher}

We now consider $\lcomet$ as a category enriched in posets, or equivalently as a \emph{locally posetal} 2-category. Concretely, we put $f \leq g$ for $f,g : X\to Y$ if and only if $f(x) \le g(x)$ for all $x\in X$. This property is preserved by $P$:

\begin{prop}\label{Pmonotone}
 Let $f\le g:X\to Y$. Then $Pf \le Pg:PX\to PY$. 
\end{prop}

\begin{proof}
 Let $h:Y\to\R$ be short and monotone. We have that for every $x\in X$, $f(x)\le g(x)$ in $Y$, therefore $h\circ f(x) \le h\circ g(x)$. Since all the measures in $PX$ are positive (or equivalently, positive linear functionals), we get that for every $p\in P$,
 \begin{equation*}
  \int_X h \, d(f_*p) = \int_X h\circ f \, dp \le \int_X h\circ g \, dp = \int_X h \, d(g_*p) .
 \end{equation*}
Since $h$ was arbitrary, we conclude $f_* p \le g_* p$ by Theorem~\ref{Ldual}. Since $p$ was arbitrary, we therefore conclude $Pf\le Pg$.
\end{proof}

\begin{cor}
 $P$ is a strict 2-functor, and so also a strict 2-monad, on $\lcomet$ (as a strict 2-category). 
\end{cor}

Consider now the adjunction given by the bijection~\eqref{choqadj2}. The operations $f\mapsto e\circ(Pf)$ and $\tilde{f}\mapsto \tilde{f}\circ\delta$ forming the bijection are monotone:
\begin{itemize}
 \item If $f \le g:X\to A$, then $Pf\le Pg$ by Proposition~\ref{Pmonotone}, and then $e\circ(Pf)\le e\circ(Pf)$ by monotonicity of $e$;
 \item If $\tilde{f} \le \tilde{g}:PX\to A$, then also $\tilde{f}\circ\delta \le \tilde{f}\circ\delta$ by monotonicity of composition.
\end{itemize}

Therefore, the correspondence
\begin{equation}\label{ordchoqadj}
\lcomet(X,A)\cong \lcomet^P(PX,A)
\end{equation} 
is not just a bijection of sets, but also an isomorphism of partial orders. In other words, it is an adjunction in the enriched (locally posetal) sense. 

From the abstract point of view, the 2-monad $P$ induces a 2-adjunction, which implies an equivalence of the hom-preorders in~\eqref{ordchoqadj}. But since all the objects of our categories are partial orders, all the hom-categories are skeletal, and therefore such an equivalence of preorders must be an isomorphism of partial orders. 

We now give a 2-categorical analogue of the concept of \emph{separation of points}. In an L-ordered space, by definition, the morphisms to $\R$ are enough to distinguish points and to determine the order. Here is how we can formalize the statement, by defining an analogue of \emph{coseparators} for locally posetal 2-categories.

\begin{deph}\label{defcosep}
 Let $\cat{C}$ be a locally posetal 2-category. We call a \emph{2-coseparator} an object $S$ of $\cat{C}$ such that the 2-functor 
 $$
 \cat{C}(-,S):\cat{C}^\op\to\cat{Poset}
 $$
 is locally fully faithful.
\end{deph}

By definition of L-orderedness, $\R$ is a 2-coseparator in the categories $\lomet$ and $\lcomet$. Conversely, we can characterize the categories $\lomet$ and $\lcomet$ as being \emph{exactly} the largest full subcategories of $\cat{OMet}$ and $\cat{COMet}$ on which $\R$ is a 2-coseparator. 

Thanks to the Hahn-Banach theorem in the form of Proposition~\ref{applyhahnbanach}, we know that the order on $P$-algebras is determined even just by \emph{affine} short monotone maps:

\begin{prop}
 Let $X\in\lcomet$, and let $A$ be a $P$-algebra. Consider two maps $f,g:X\to A$. Then $f\le g$ in the pointwise order if and only if for every $P$-morphism $h:A\to\R$, we have $h\circ f \le h\circ g$.
\end{prop}

\begin{proof}
 Since $h$ is required to be monotone, the ``only if'' direction is trivial. 
 
 Suppose now that $f\nleq g$. Then by definition there exists $x\in X$ such that $f(x)\nleq g(x)$ in $A$. By Proposition~\ref{applyhahnbanach}, we know that there exists an affine map $h:A\to\R$ such that $h(f(x))>h(g(x))$ strictly, So $h\circ f \nleq h\circ g$. 
\end{proof}

\begin{cor}\label{cosepalg}
 The real line $\R$ is a 2-coseparator in the Eilenberg-Moore category $\lcomet^P$.
\end{cor}

\subsection{Convex monotone maps as oplax morphisms}\label{ssecoplax}

In this subsection, we still consider $\lcomet$ a strict 2-category and $P$ a strict 2-monad, in the sense explained in~\ref{ssechigher}.

This means that for algebras of the ordered Kantorovich monad, the algebra morphisms are not the only interesting maps: there are also \emph{lax} algebra morphisms.
A lax $P$-morphism $f:A\to B$ is a short, monotone map together with a 2-cell (which here is a property rather than a structure),
\begin{equation}\label{laxPmap}
 \begin{tikzcd}
  PA \ar[swap]{d}{e} \ar{r}{Pf} 
   \ar[bend left,""{name=UPRI,pos=0.55},phantom]{dr} \ar[bend right,""{name=DOLE,pos=0.45},phantom]{dr}
    \ar[Rightarrow, from=UPRI, to=DOLE]
   & PB \ar{d}{e} \\
  A \ar[swap]{r}{f} & B 
 \end{tikzcd}
\end{equation}
Spelling this out, the condition is that $e(f(p)) \le f(e(p))$ for all $p\in PX$.

These maps are well known, at least in a special case.

\begin{prop}\label{concavemaps}
 Let $A$ be an unordered $P$-algebra, and consider $\R$ with its usual order. Let $f:A\to\R$ be short (and automatically monotone). Then $f$ is a lax $P$-morphism if and only if it is a concave function. 
\end{prop}

\begin{proof}
 Diagram \eqref{laxPmap} can be written explicitly as:
 \begin{equation}
  \int_A f(a) \, dp(a) \le f \left( \int_A a\, dp(a) \right)
 \end{equation}
 for any $p\in PA$. 
By the generalized Jensen's inequality, this is equivalent to
 \begin{equation}
  \lambda\,f(a) + (1-\lambda)\,f(b) \;\le\; f\big( \lambda\,a + (1-\lambda)\,b \big) 
 \end{equation}
 for all $a,b\in A$ and $\lambda\in[0,1]$. 
 This is the usual definition of a concave function. 
\end{proof}

More in general, we think of lax $P$-morphisms as \emph{monotone} concave functions. Dually, \emph{oplax} $P$-morphisms---which are as in~\eqref{laxPmap} but with the inequality pointing the opposite way---correspond to monotone convex functions. We therefore have the following categories:
\begin{itemize}
 \item $P\cat{Alg}_s$, the category of $P$-algebras and \emph{strict} $P$-morphisms (affine maps);
 \item $P\cat{Alg}_l$, the category of $P$-algebras and \emph{lax} $P$-morphisms (concave maps);
 \item $P\cat{Alg}_o$, the category of $P$-algebras and \emph{oplax} $P$-morphisms (convex maps).
\end{itemize}
All these categories are again locally posetal 2-categories, and since they contain all affine maps, they all admit $\R$ with its usual order as a 2-coseparator.

Using Theorem~\ref{poscone}, we have then proven the following:

\begin{thm}\label{laxoplax}
Consider the monad $P$ on $\lcomet$. Then:
\begin{itemize}
 \item $P\cat{Alg}_s$ is equivalent to the category of closed convex subsets $A\subseteq E$ with $E$ an ordered Banach space, with morphisms given by monotone \emph{affine} short maps;
 \item $P\cat{Alg}_l$ is equivalent to the category of closed convex subsets $A\subseteq E$ with $E$ an ordered Banach space, with morphisms given by monotone \emph{concave} short maps;
 \item $P\cat{Alg}_o$ is equivalent to the category of closed convex subsets $A\subseteq E$ with $E$ an ordered Banach space, with morphisms given by monotone \emph{convex} short maps.
\end{itemize}
\end{thm}

\begin{remark}
 It is a very well-known fact that the composition $f\circ g$ of two convex functions $f,g:\R\to\R$ may not be a convex function; and that if $f$ is in addition monotone, then $f\circ g$ \emph{is} convex. We now explain how this makes perfect sense within our framework. We write $(\R,\le)$ for the object $\R\in\lcomet$ equipped with its usual order, and $(\R,=)$ for $\R\in\lcomet$ equipped with the discrete order. Technically all our maps are assumed to be short, but the same considerations should apply more generally.

By Proposition \ref{concavemaps}, a concave function $\R\to\R$ is the same thing as a lax $P$-morphism $(\R,=) \to (\R,\le)$; monotonicity is a trivial requirement. A monotone concave function $\R\to\R$ is the same thing as a lax $P$-morphism $(\R,\le) \to (\R,\le)$. In our formalism, both functions are technically monotone, but with respect to different orders on the domain. Due to the possibility of composing in $P\cat{Alg}_l$, we have:
 \begin{itemize}
  \item Two concave monotone functions $(\R,\le) \to (\R,\le)$ can be composed, giving again a concave monotone function $(\R,\le) \to (\R,\le)$;
  \item A concave monotone function $(\R,=) \to (\R,\le)$ can be postcomposed with a concave monotone function $(\R,\le) \to (\R,\le)$, giving a concave monotone function $(\R,=) \to (\R,\le)$;
  \item Two concave monotone functions $(\R,=) \to (\R,\le)$ \emph{cannot} be composed, since domain and codomain do not match.
 \end{itemize}
 We see that in this framework, the rule for when the composition of concave functions is again concave is just elementary category theory. The same applies to convex functions as oplax $P$-morphisms.
\end{remark}

\newpage
\appendix
\counterwithin{thm}{section} 

\section{The metric lifting property}
\label{appmlift}

In this appendix we define the metric lifting property used in the proofs of~\ref{ssecproperness}. 

\begin{deph}\label{defmetriclifting}
A short map $f : X \to Y$ in $\cat{Met}$ has the \emph{metric lifting property} if for every $C > 0$ and $x\in X$ and $y'\in Y$ with $d(f(x),y') < C$, there is $x'\in X$ with $f(x') = y'$ and $d(x,x') < C$.
\end{deph}

This notion is similar to the notion of \emph{submetry}~\cite{submetries}. A map $f:X\to Y$ between metric spaces is called a submetry if for every $x\in X$ and every $r>0$, the closed ball $B(x,r)$ of radius $r$ centered at $x$ is mapped by $f$ to the ball with the same radius, centered at $f(x)$. That is, 
$$
f(B(x,r))\;=\;B(f(x),r).
$$
If one replaces the strict inequalities ($<$) in the definition of metric lifting property with lax inequalities ($\le$), one recovers the notion of submetry. However, in general the metric lifting property is a slightly weaker condition, as the following example shows.

\begin{eg}\label{not_submetry_example}
 Let $X$ be a metric spaces constructed as follows. It has $\N$-many points, which we label as $w,x_1,x_2,\dots$, and the following distances:
 \begin{itemize}
  \item For all $i>j\ge 1$, $d(x_i,x_j)\coloneqq 2$;
  \item For all $i\ge 1$, $d(w,x_i)\coloneqq 1 +\frac{1}{n}$.
 \end{itemize}
 Consider also the space $Y\coloneqq \{y,z\}$, with $d(y,z)=1$, and the map $f:X\to Y$ which maps $w\mapsto y$ and $x_i\mapsto z$ for all $i$. We sketch the situation in the picture below.
 The map $f$ satisfies the metric lifting property, but is not a submetry.

 
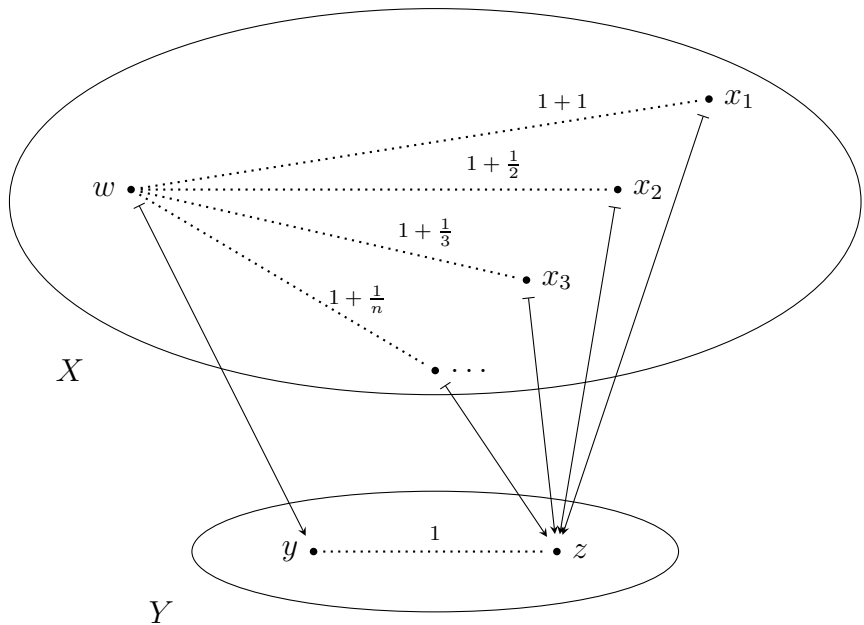
\begin{center}
  \begin{tikzpicture}[baseline=(current  bounding  box.center),>=stealth,scale=0.8]
  
  \node[bullet, label=left:$w$] (w) at (0,4) {};
  \node[bullet, label=right:$x_1$] (x1) at (9.5,5.5) {};
  \node[bullet, label=right:$x_2$] (x2) at (8,4) {};
  \node[bullet, label=right:$x_3$] (x3) at (6.5,2.5) {};
  \node[bullet, label=right:$\dots$] (x4) at (5,1) {};
  
  \draw (5,3.8) ellipse (7cm and 3.2cm);
  \node[label=center:$X$] (X) at (-1,1) {};
  
  \node[bullet, label=left:$y$] (y) at (3,-2) {};
  \node[bullet, label=right:$z$] (z) at (7,-2) {};
  
  \draw (5,-2) ellipse (4cm and 1cm);
  \node[label=center:$Y$] (Y) at (.5,-3) {};
  
  \draw[relation, dotted] (w) -- (x1) node[near end,above,font=\fontsize{8}{0}\selectfont] {$1+1$};
  \draw[relation, dotted] (w) -- (x2) node[near end,above,font=\fontsize{8}{0}\selectfont] {$1+\frac{1}{2}$};
  \draw[relation, dotted] (w) -- (x3) node[near end,above,font=\fontsize{8}{0}\selectfont] {$1+\frac{1}{3}$};
  \draw[relation, dotted] (w) -- (x4) node[near end,above,font=\fontsize{8}{0}\selectfont] {$1+\frac{1}{n}$};;
  
  \draw[|->, shorten <= 5pt, shorten >=5pt] (w) -- (y);
  \draw[|->, shorten <= 5pt, shorten >=5pt] (x1) -- (z);
  \draw[|->, shorten <= 5pt, shorten >=5pt] (x2) -- (z);
  \draw[|->, shorten <= 5pt, shorten >=5pt] (x3) -- (z);
  \draw[|->, shorten <= 5pt, shorten >=5pt] (x4) -- (z);
  
  \draw[relation, dotted] (y) -- (z) node[midway,above,font=\fontsize{8}{0}\selectfont] {$1$};
 \end{tikzpicture}
\end{center}
\end{eg}

Here is a result which allows us to extend the metric lifting property from a dense subset to the whole space, used in the proof of Lemma~\ref{deltametriclifting}.

\begin{lemma}
Let $f : X\to Y$ be in $\cat{Met}$ with $X$ complete, and $D\subseteq X$ and $E\subseteq Y$ dense. Suppose that for every $x\in D$, $y'\in E$ and $C>0$ with $d(f(x), y') < C$ there is $x'\in D$ with $f(x') = y'$ and $d(x,x') < C$. Then $f$ has the metric lifting property.
\label{metricliftinggeneral}
\end{lemma}

We can illustrate the situation as follows:

 
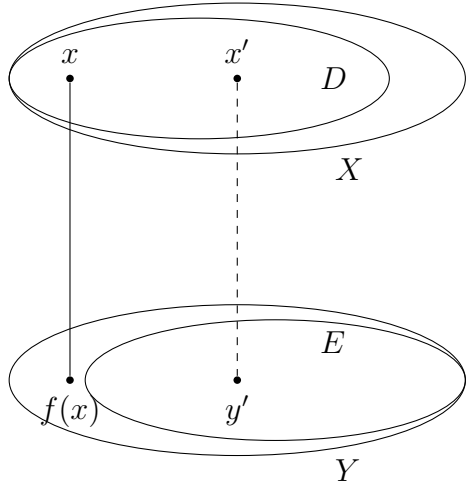
\begin{center}
 \begin{tikzpicture}[baseline=(current  bounding  box.center),>=stealth]
   \draw (0,4) ellipse (3cm and 1cm);
   \node[label=right:$X$] (x) at (1,2.8) {};
   
   \draw (-0.5,4) ellipse (2.5cm and 0.8cm);
   \node[label=right:$D$] (x) at (0.8,4) {};
   
   \draw (0,0) ellipse (3cm and 1cm);
   \node[label=right:$Y$] (x) at (1,-1.2) {};
   
   \draw (0.5,0) ellipse (2.5cm and 0.8cm);
   \node[label=right:$E$] (x) at (0.8,0.5) {};
   
   \node[bullet,label=above:$x$] (x) at (-2.2,4) {};
   \node[bullet,label=above:$x'$] (x') at (0,4) {};
   \node[bullet,label=below:$f(x)$] (fx) at (-2.2,0) {};
   \node[bullet,label=below:$y'$] (y') at (0,0) {};
   
   \draw[-,dashed] (x') -- (y');
   \draw[-] (x) -- (fx);
 \end{tikzpicture}
 \end{center}

\begin{proof}
Suppose that we are given $x\in X$ and $y'\in Y$ with $d(f(x),y') < C$ and look for a corresponding $x'\in X$. Since the inequality $d(f(x),y') < C$ is strict, we can find $\e > 0$ such that $d(f(x),y') < C - 2\e$ is still true. Then choose $\bar{x}\in D$ with $d(x,\bar{x}) < \e$, so that $d(f(\bar{x}),y') < C - \e$, using the triangle inequality and shortness of $f$. We furthermore choose $\nu > 0$ such that
\[
	d(f(\bar{x}),y') < C - \e - 2\nu
\]
is still valid. Now we approximate $y'$ by a sequence $\{\bar{y}'_j\}$ in $E$, starting at $j=1$, such that $d(y',\bar{y}'_j) < 2^{-(j+1)}\nu$, which guarantees $d(\bar{y}'_j,\bar{y}'_{j+1}) < 2^{-j} \nu$, and in particular
\begin{align*}
	d(f(\bar{x}),\bar{y}'_0) & \leq d(f(\bar{x}),y') + d(y',\bar{y}'_0) \\
		& < (C - \e - \; 2\nu) + \nu \leq C - \e - \; \nu.
\end{align*}
Then we construct a sequence $\{\bar{x}'_j\}$ in $D$, starting at $j=1$ as well, with the properties
\begin{itemize}
\item $f(\bar{x}'_j) = \bar{y}'_j$, and
\item $d(\bar{x}'_j,\bar{x}'_{j+1}) < 2^{-j}\nu$.
\end{itemize}
We can achieve this by taking $\bar{x}'_0$ to be the lift of $\bar{y}'_0$ relative to $\bar{x}$, and then apply the lifting assumption again repeatedly with respect to $\bar{x}'_j$ and $\bar{y}'_{j+1}$, resulting in $\bar{x}'_{j+1}$. A standard geometric series argument shows that $\{\bar{x}'_j\}$ is a Cauchy sequence, and we denote its limit by $x'$. By the first item above and continuity of $f$, we have $f(x') = \lim_j \bar{y}'_j = y'$, as desired. Moreover,
\begin{align*}
	d(x,x') = \lim_j d(x,\bar{x}'_j) & \leq d(x,\bar{x}) + d(\bar{x},\bar{x}'_0) + \sum_j d(\bar{x}'_j,\bar{x}'_{j+1}) \\
		& < \e \, + \; ( C - \e - \; \nu ) \, + \nu \\
		& = C,
\end{align*}
as was to be shown.
\end{proof}

Here is another useful technical result, used in the proof of Theorem~\ref{deltaproper}.

\begin{lemma}\label{pointtoproper}
Suppose that $f : X\to Y$ in $\cat{CMet}$ has the metric lifting property, and that $f$ is such that preimages of points are compact sets. Then $f$ is proper, i.e., preimages of compact sets are compact.
Moreover, $f$ is a submetry.
\end{lemma}

\begin{proof}
We start with properness. Since $f$ is by assumption continuous, we only need to prove that preimages of compact sets are totally bounded. Let now $K\subseteq Y$ be compact. For every $\e>0$, there exists a finite $(\e/2)$-net $\{y_n\}$ covering $K$ (i.e.~every $k\in K$ is within distance $\e/2$ from $\{y_n\}$). Take now the finite collection of sets $\{f^{-1}(y_n)\}$. By hypothesis, we know that all those sets are compact. Moreover, by the metric lifting property, we know that every element $x\in f^{-1}(K)$ is within distance $\e/2$ from some element of $\bigcup_n f^{-1}(y_n)$. Now the set $\bigcup_n f^{-1}(y_n)$ is a finite union of compact sets, so it is compact, and in particular it can be covered by finitely many balls of radius $\e/2$. This implies that for every $\e>0$, the whole $f^{-1}(K)$ can be covered by finitely many balls of radius $\e$, i.e.~it is totally bounded. 

Now to prove that $f$ is a submetry, let $x\in Y$ and $y'\in Y$. We have to find $x'\in X$ such that $d(x',x) = d(y',f(x))$. By metric lifting, we can find a sequence $\{x'_n\}$ in $X$ such that $d'(x'_n,x)$ converges to $d(y',f(x))$. Now choose a convergent subsequence by compactness.
\end{proof}

\section{Relationship with Lawvere metric spaces}\label{secindlmet}

In this final section, we connect our treatment of L-ordered metric spaces to the theory of Lawvere metric spaces~\cite{lawvere,seriously,gl}, which turns out to be intimately related. In particular, we will see that the L-distance defined in Section~\ref{secLord} can be interpreted as a particular Lawvere distance associated to the metric and the order. This also explains where the ``L'' in ``L-ordered'' comes from.

\begin{deph}
 A \emph{Lawvere metric space}~\cite{lawvere}, or more briefly LMS, is a set $\mathcal{X}$ equipped with a function $d_{L}:\mathcal{X}\times \mathcal{X}\to [0,\infty]$, called \emph{Lawvere metric}, such that:
\begin{enumerate}
 \item\label{dzero} $d_{L}(x,x)= 0$ for all $x\in\mathcal{X}$;
 \item\label{dtrans} $d_{L}(x,z)\le d_{L}(x,y) + d_{L}(y,z)$ for all $x,y,z\in\mathcal{X}$.
\end{enumerate} 
\end{deph}

In this work, we will for convenience restrict to Lawvere metrics which assume only finite values, $d_L(x,y) < \infty$.

Given LMSs $\mathcal{X}$ and $\mathcal{Y}$, a function $f:\mathcal{X}\to\mathcal{Y}$ is called a \emph{LMS morphism} if and only if it is short for the Lawvere metric: for every $x,x'\in\mathcal{X}$,
\begin{equation*}
 d_{L}(f(x),f(x')) \le d_{L}(x,x') .
\end{equation*}

Every ordinary metric space is a Lawvere metric space. Conversely, in general a Lawvere metric does not need to be symmetric, but one can obtain a (pseudo)metric by symmetrization:
\begin{equation}\label{symmet}
 d(x,y) := \max \{d_{L}(x,y), d_{L}(y,x) \}.
\end{equation}

Every preorder can be given a canonical Lawvere metric, as:
\begin{equation*}
 d_{L}(x,y) := \begin{cases}
  0 & x\le y ; \\
  1 & x\nleq y .
 \end{cases}
\end{equation*}
or even $\infty$ instead of $1$ (if we allowed that value).
Since a preorder is reflexive and transitive, the properties~\ref{dzero} and~\ref{dtrans} are automatically satisfied. Conversely, a Lawvere metric $d_L$ induces canonically a preorder by the following assignment:
\begin{equation}\label{sympord}
 x\le y  \quad\mbox{if and only if}\quad d_{L}(x,y)=0 .
\end{equation}

We can combine both constructions to obtain an \emph{ordered} metric space:

\begin{deph}
 Let $\mathcal{X}$ be a LMS. We define the ordered metric space $S\mathcal{X}$ to be:
 \begin{itemize}
  \item As set, the underlying set of $\mathcal{X}$, but where we identify any two points whenever \emph{both} their distances vanish;
  \item As metric $d$, the symmetrized metric~\eqref{symmet};
  \item As order $\le$, the one defined by~\eqref{sympord}.
 \end{itemize}
\end{deph}

By construction, we have that:
\begin{itemize}
 \item The induced pseudometric $d$ is actually a metric;
 \item The induced preorder $\le$ is actually a partial order. 
\end{itemize}

\begin{prop}
 Let $\mathcal{X}$ be a LMS, and construct its metric $d$ by symmetrization of $d_{L}$. Then $d_{L}$ is short with respect to the metric $d : \mathcal{X} \otimes \mathcal{X} \to \R$, and in particular continuous.
\end{prop}

\begin{proof}
It is sufficient prove shortness in each argument separately. In the first argument,
 \begin{align*}
  | d_{L}(x,z) - d_{L}(y,z) | \le \max \{d_{L}(x,y), d_{L}(y,x) \} = d(x,y),
 \end{align*}
thanks to the triangle inequality for $d_L$. Similarly in the second argument,
 \begin{align*}
  | d_{L}(z,x) - d_{L}(z,y) | \le \max \{d_{L}(x,y), d_{L}(y,x) \} = d(x,y).
 \end{align*}
\end{proof}

\begin{prop}
 Let $\mathcal{X}$ be a LMS. Then the order on $S\mathcal{X}$ is closed. Therefore $S\mathcal{X}\in\cat{OMet}$.
\end{prop}

\begin{proof}
 Let $\{x_i\}$ and $\{y_i\}$ be sequences in $S\mathcal{X}$ converging to $x$ and $y$, respectively, and such that for all $i$, we have the inequality $x_i\le y_i$. Then, by definition of the induced order, $d_{L}(x_i,y_i)=0$ for all $i$. Since $d_{L}$ is continuous, this implies $d_{L}(x,y)=0$, which in turn gives $x\le y$. Therefore the order is closed.
\end{proof}

\begin{prop}
 Let $\mathcal{X}$ and $\mathcal{Y}$ be LMS. Given a LMS morphism $f:\mathcal{X}\to\mathcal{Y}$, the induced map $Sf:S\mathcal{X}\to S\mathcal{Y}$ (given by the function $f$ on the underlying sets) is short and monotone. 
\end{prop}

\begin{proof}
 Since $f$ is a LMS morphism, we have that for every $x,x'\in\mathcal{X}$,
 \begin{equation*}
  d_{L}(f(x),f(x'))\le d_{L}(x,x').
 \end{equation*}
 Therefore,
 \begin{align*}
  d((Sf)(x),(Sf)(x')) &= \max \{d_{L}(f(x),f(x')), d_{L}(f(x'),f(x)) \} \\ &\le \max \{d_{L}(x,x'), d_{L}(x',x) \} = d(x,x') .
 \end{align*}
 Moreover,
 \begin{align*}
  x\le x' \quad\Rightarrow\quad d_{L}(x,x') = 0 \quad\Rightarrow\quad d_{L}(f(x),f(x')) = 0 \quad\Rightarrow\quad (Sf)(x) \le (Sf)(x') .
 \end{align*}
 Therefore $Sf$ is short and monotone. 
\end{proof}

We write $\cat{LMet}$ the category of LMSs and LMS morphisms between them. We have shown:

\begin{cor}
The assignment $\mathcal{X}\mapsto S\mathcal{X}, f\mapsto Sf$ is a functor $S:\cat{LMet}\mapsto\cat{OMet}$. 
\end{cor}

We have seen how to obtain an ordered metric space from a Lawvere metric space by means of the symmetrization functor $S$. Conversely, given an ordered metric space $X$, one can define a Lawvere metric space in a universal way, by means of the L-distance introduced in Definition~\ref{defLdist}. 
We know by Proposition~\ref{dlproperties} that $d_L$ is indeed a Lawvere metric. Moreover, by Remark~\ref{smL}, any short monotone map $f:X\to Y$ is also short for the L-distance, in the LMS sense. 

\begin{deph}
 Let $X,Y$ be a ordered metric spaces, and $f:X\to Y$ be short and monotone.
 \begin{enumerate}
  \item We denote by $LX$ the Lawvere metric space given by $X$ together with the Lawvere distance $d_L$ of Definition~\ref{defLdist}. 
  \item We denote by $Lf$ the map $LX\to LY$ induced by $f$, which is a LMS morphism by Remark~\ref{smL}.
 \end{enumerate}
\end{deph}

The assignment $X\to LX$, $f\to LF$ is therefore a functor $L:\cat{OMet}\mapsto\cat{LMet}$, which we call the \emph{L-functor}.

\begin{eg}
 Consider $\R$ with its usual metric and order\footnote{So $\R$ is considered ordered \emph{upwards}, differently from $[0,\infty]$ in~\cite{lawvere}.}. Then the induced Lawvere metric ${L}\R$ is given by:
 \begin{equation}
  d_{L}(x,y) = \begin{cases}
                        0 & x\le y  ;\\
                        |x-y| & x\nleq y .
                       \end{cases}
 \end{equation}
 Intuitively, going upwards in the order has zero cost, but going down has cost equal to the usual distance. 
 It is easy to check that $SL\R\cong\R$. So in particular, a map $X\to \R$ is short and monotone if and only if it is short and monotone as a map $X\to L\R$, in the generalized sense given above. This is true for $\R$, but not for all spaces.
\end{eg}

\begin{thm}\label{LadjS}
 The L-functor $L:\cat{OMet}\mapsto\cat{LMet}$ is left adjoint to the symmetrization functor $S:\cat{LMet}\mapsto\cat{OMet}$.
\end{thm}

The proof is below. We thus have a natural bijection,
\begin{equation}\label{LSadj}
 \cat{LMet}({L}X,\mathcal{Y}) \cong \cat{OMet}(X,S\mathcal{Y}) ,
\end{equation}
between LMS morphisms ${L}X\to\mathcal{Y}$ and short monotone maps $X\to S\mathcal{Y}$.
Now with a slight abuse, let's call a function $f:X\to \mathcal{Y}$:
 \begin{itemize}
  \item \emph{short}, if for every $x,x'\in X$,
  \begin{equation*}
   d_{L}(f(x),f(x')) \le d(x,x') ;
  \end{equation*}
  \item \emph{monotone}, if if for every $x,x'\in X$,
  \begin{equation*}
   x\le x' \quad \Rightarrow \quad d_{L}(f(x),f(x')) = 0 .
  \end{equation*}
 \end{itemize}
Then $f:X\to \mathcal{Y}$ is short and monotone in the sense given above if and only if it is short and monotone as a map $X\to S\mathcal{Y}$. Moreover, given an ordered metric space $X$, the natural map $\ell:X\to LX$ induced by the identity on the underlying sets is short and monotone in the sense given above. 
The bijection~\eqref{LSadj} means that $LX$ and $\ell:X\to {L}X$ satisfy the following universal property: for every LMS $\mathcal{Y}$ and for every short, monotone map $f:X\to \mathcal{Y}$ in the sense given above, there exists a unique LMS morphism ${L}X\to \mathcal{Y}$ making the diagram commute:
 \begin{equation}\label{uniindlms}
  \begin{tikzcd}
   X \ar{d}[swap]{\ell} \ar{dr}{f} \\
   {L}X \uni{r} & \mathcal{Y}
  \end{tikzcd}
 \end{equation}
The space ${L}X$ can thus be interpreted as the largest LMS such that its Lawvere metric is still compatible with the metric of $X$ and with its order.

In order to prove the theorem, we need a useful lemma, which tells us that the Lawvere metric satisfies an equation similar to the L-distance of Definition~\ref{defLdist}:
\begin{lemma}\label{alreadyl}
 Let $\mathcal{X}$ be a LMS. Then for every $x,y\in \mathcal{X}$,
 \begin{equation*}
  d_L(x,y) = \sup_{f:\mathcal{X}\to L\R} \big( f(x) - f(y) \big),
 \end{equation*}
 where the supremum is taken over all LMS morphisms $\mathcal{X}\to L\R$. 
\end{lemma}

This is essentially the enriched Yoneda lemma, and the proof is correspondingly similar.

\begin{proof}
 First of all, for every LMS morphism $f:\mathcal{X}\to\R$, 
 \begin{align*}
  f(x) - f(y)\quad  \le \quad d_L(f(x),f(y)) \quad \le \quad d_L(x,y).
 \end{align*} 
 Conversely, by the triangle inequality, for every $x,y,z\in \mathcal{X}$ we have
 \begin{align*}
  d_L(x,z) - d_L(y,z) \quad \le \quad d_L(x,y) ,
 \end{align*}
 which means that $d_L(-,z):\mathcal{X}\to\R$ is a LMS morphism for every $z\in \mathcal{X}$. By setting $f(-) := d_L(-,y)$, we then have that 
 \begin{align*}
  d_L(x,y) \quad = \quad d_L(x,y) - d_L(y,y) \quad = \quad f(x) - f(y).
 \end{align*} 
\end{proof}

\begin{proof}[Proof of Theorem~\ref{LadjS}]
 Consider the diagram~\eqref{uniindlms}. There exists a unique \emph{function} ${L}X\to \mathcal{Y}$ making the diagram commute, namely the one which is equal to $f$ on the underlying set. Let's call this map $\tilde{f}$. We have to prove that $\tilde{f}$ is a LMS morphism, i.e.~that for each $x,y\in LX$,
 $$
  d_{L}(\tilde{f}(x),\tilde{f}(y)) \le d_L(x,y),
 $$
 or equivalently that for each $x,y\in X$,
 $$
  d_{L}(f(x),f(y)) \le \sup_{u:X\to\R} \big( u(x) - u(y) \big),
 $$
 with the supremum taken over all short, monotone maps. 
 
 Now using Lemmas~\ref{alreadyl}, we have
 \begin{align*}
  d_{L}(f(x),f(y)) &= \sup_{g:\mathcal{Y}\to L\R} \big( g\circ f(x) - g\circ f(y) \big) .
 \end{align*}
 The map $g\circ f$ is the composite of a short monotone map to $L\R$ (in the generalized sense) and a LMS morphism, so it is a short monotone map $X\to L\R$, or equivalently, a short and monotone map $X\to\R$ in the usual sense. 
 We are then left with
 \begin{align*}
  d_{L}(f(x),f(y)) &= \sup_{g:\mathcal{Y}\to L\R} \big( g\circ f(x) - g\circ f(y) \big) \\
   &\le \sup_{u:X\to L\R} \big( u(x) - u(y) \big) \\
   &= \sup_{u:X\to \R} \big( u(x) - u(y) \big) .
 \end{align*}
 with the supremum taken over all short monotone maps.
\end{proof}

Therefore, as we have seen, the L-distance of Definition~\ref{defLdist} can be interpreted as the universal Lawvere metric induced by the ordered metric space structure. L-ordered spaces, then, are the spaces for which this induced Lawvere metric is enough to determine the order.

\paragraph{Acknowledgements.} We thank Rostislav Matveev and Sharwin Rezagholi for the helpful discussions. We also want to thank the anonymous reviewer for the helpful questions and comments.

\cleardoublepage
\bibliographystyle{alpha}
\bibliography{catprob}

\newcommand{\etalchar}[1]{$^{#1}$}
\begin{thebibliography}{GHK{\etalchar{+}}03}

\bibitem[AT07]{conesduality}
Charalambos~D. Aliprantis and Rabee Tourky.
\newblock {\em {Cones and duality}}, volume~84 of {\em {Graduate Studies in
  Mathematics}}.
\newblock American Mathematical Society, Providence, RI, 2007.

\bibitem[Bas15]{hitch}
Giuliano Basso.
\newblock {A Hitchhiker's guide to {Wasserstein} distances}.
\newblock
  \href{http://n.ethz.ch/~gbasso/download/A\%20Hitchhikers\%20guide\%20to\%20Wasserstein/A\%20Hitchhikers\%20guide\%20to\%20Wasserstein.pdf}{Available
  at: http://n.ethz.ch}, 2015.

\bibitem[Ber87]{submetries}
V.~N. Berestovskiĭ.
\newblock ``{S}ubmetries'' of three-dimensional forms of nonnegative curvature.
\newblock {\em Sibirsk. Mat. Zh.}, 28(4):44--56, 1987.
\newblock In Russian.

\bibitem[CF13]{capraro}
Valerio Capraro and Tobias Fritz.
\newblock {On the axiomatization of convex subsets of {Banach} spaces}.
\newblock {\em Proc. Amer. Math. Soc.}, 141, 2013.
\newblock \href{https://arxiv.org/abs/1105.1270}{arXiv:1105.1270}.

\bibitem[Edw78]{edwards}
David~A. Edwards.
\newblock {On the existence of probability measures with given marginals}.
\newblock {\em Ann. Inst. Fourier (Grenoble)}, 28(4):53--78, 1978.

\bibitem[Fis80]{fishburn}
Peter~C. Fishburn.
\newblock {Stochastic dominance and moments of distributions}.
\newblock {\em Math. Oper. Res.}, 5(1):94–100, 1980.

\bibitem[FP18]{ours_bimonoidal}
Tobias Fritz and Paolo Perrone.
\newblock {Bimonoidal Structure of Probability Monads}.
\newblock {\em Proceedings of MFPS 34, ENTCS}, 2018.
\newblock \href{https://arxiv.org/abs/1804.03527}{arXiv:1804.03527}.

\bibitem[FP19]{ours_kantorovich}
Tobias Fritz and Paolo Perrone.
\newblock A probability monad as the colimit of spaces of finite samples.
\newblock {\em Theory and Applications of Categories}, 34(7):170--220, 2019.
\newblock \href{https://arxiv.org/abs/1712.05363}{arXiv:1712.05363}.

\bibitem[Fri19]{antisymmetry}
Tobias Fritz.
\newblock Antisymmetry of the stochastic order on all ordered topological
  spaces.
\newblock {\em Anal. Geom. Metr. Spaces}, 7(1):250--252, 2019.
\newblock \href{https://arxiv.org/abs/1810.06771}{arXiv:1810.06771}.

\bibitem[GHK{\etalchar{+}}03]{continuous}
G.~Gierz, K.~H. Hofmann, K.~Keimel, J.~D. Lawson, M.~W. Mislove, and D.~Scott.
\newblock {\em {Continuous Lattices and Domains}}.
\newblock Cambridge University Press, 2003.

\bibitem[Gir82]{giry}
Michèle Giry.
\newblock {A Categorical Approach to Probability Theory}.
\newblock In {\em {Categorical aspects of topology and analysis}}, volume 915
  of {\em {Lecture Notes in Mathematics}}. 1982.

\bibitem[GL17]{gl}
Jean Goubault-Larrecq.
\newblock {Complete Quasi-Metrics for Hyperspaces, Continuous Valuations, and
  Previsions}, 2017.
\newblock \href{arXiv:1707.03784}{arXiv:1707.03784}.

\bibitem[HLL18]{hll}
Fumio Hiai, Jimmie Lawson, and Yongdo Lim.
\newblock {The stochastic order of probability measures on ordered metric
  spaces}.
\newblock {\em Journal of Mathematical Analysis and Applications},
  464(1):707--724, 2018.
\newblock \href{https://arxiv.org/abs/1709.04187}{arXiv:1709.04187}.

\bibitem[JP89]{jones-plotkin}
C.~Jones and J.~D. Plotkin.
\newblock {A Probabilistic Powerdomain of Evaluations}.
\newblock {\em Proceedings of the Fourth Annual Symposium of Logics in Computer
  Science}, 1989.

\bibitem[Kei08]{keimel}
Klaus Keimel.
\newblock {The monad of probability measures over compact ordered spaces and
  its {E}ilenberg-{M}oore algebras}.
\newblock {\em Topology Appl.}, 156(2):227–239, 2008.

\bibitem[Kel84]{kellerer}
Hans~G. Kellerer.
\newblock {Duality Theorems for Marginal Problems}.
\newblock {\em Zeitschrift f\"ur Warscheinlichkeitstheorie und verwandte
  Gebiete}, 67:399--432, 1984.

\bibitem[Law62]{early}
William Lawvere.
\newblock {The category of probabilistic mappings}.
\newblock Available at
  \href{https://ncatlab.org/nlab/files/lawvereprobability1962.pdf}{https://ncatlab.org/nlab/files/lawvereprobability1962.pdf},
  1962.

\bibitem[Law73]{lawvere}
William Lawvere.
\newblock {Metric spaces, generalized logic and closed categories}.
\newblock {\em Rendiconti del seminario matematico e fisico di Milano}, 43,
  1973.

\bibitem[Law86]{seriously}
William Lawvere.
\newblock {Taking Categories Seriously}.
\newblock {\em Revista Colombiana de Matematicas}, 20, 1986.

\bibitem[Law17]{lawson}
Jimmie Lawson.
\newblock {Ordered probability spaces}.
\newblock {\em J. Math. Anal. Appl.}, 455(1):167--179, 2017.
\newblock \href{https://arxiv.org/abs/1612.03213}{arXiv:1612.03213}.

\bibitem[Leh55]{lehmann}
E.~L. Lehmann.
\newblock {Ordered Families of Distributions}.
\newblock {\em Annals of Mathematical Statistics}, 26(3):399–419, 1955.

\bibitem[{Mac}00]{maclane}
Saunders {Mac Lane}.
\newblock {\em {Categories for the Working Mathematician}}.
\newblock Springer, 2000.

\bibitem[Nac65]{nachbin}
Leopold Nachbin.
\newblock {\em {Topology and Order}}.
\newblock Van Nostrand, 1965.

\bibitem[Par05]{partha}
K.~R. Parthasarathy.
\newblock {\em Probability measures on metric spaces}.
\newblock AMS Chelsea Publishing, Providence, RI, 2005.
\newblock Reprint of the 1967 original.

\bibitem[SS07]{stochastic-orders}
Moshe Shaked and George Shanthikumar.
\newblock {\em {Stochastic Orders}}.
\newblock Springer, 2007.

\bibitem[Str65]{strassen}
Volker Strassen.
\newblock {The existence of probability measures with given marginals}.
\newblock {\em Annals of Mathematical Statistics}, 36:423--439, 1965.

\bibitem[vB05]{breugel}
Franck van Breugel.
\newblock {The Metric Monad for Probabilistic Nondeterminism}.
\newblock
  \href{http://www.cse.yorku.ca/~franck/research/drafts/monad.pdf}{Available at
  http://www.cse.yorku.ca}, 2005.

\bibitem[Vil09]{villani}
Cédric Villani.
\newblock {\em {Optimal transport: old and new}}, volume 338 of {\em
  {Grundlehren der mathematischen Wissenschaften}}.
\newblock Springer, 2009.

\bibitem[Win85]{winkler}
Gerhard Winkler.
\newblock {\em {Choquet order and simplices with applications in probabilistic
  models}}.
\newblock {Lecture Notes in Mathematics}. Springer, 1985.

\end{thebibliography}
\addcontentsline{toc}{section}{\bibname}

\end{document}